%% file: Main.tex
\documentclass[12pt]{elsarticle}

\usepackage{amssymb}
\usepackage{amsmath}
\usepackage{algorithm}
\usepackage{xcolor}

\usepackage{amsthm}
\usepackage{amsfonts}

\usepackage[colorlinks=true]{hyperref}
\usepackage[margin=2.25cm]{geometry}
\journal{Linear Algebra and its Applications}
\input{macro.tex}

\input{header.tex}

\usepackage{cleveref}
\begin{document}

\begin{frontmatter}



\title{Riemannian Newton optimization methods for the symmetric tensor %rank
    approximation problem}


\author[inst1,inst2]{Rima Khouja\corref{mycorrespondingauthor}}
\cortext[mycorrespondingauthor]{Corresponding author}
\ead{rima.khouja@inria.fr}

\affiliation[inst1]{organization={Inria d’Université Côte d’Azur},%Department and Organization
            addressline={Aromath}, 
            city={Sophia Antipolis},
            %postcode={}, 
            %state={},
            country={France}}

\author[inst2]{Houssam Khalil}
\author[inst1]{Bernard Mourrain}

\affiliation[inst2]{organization={Laboratory of Mathematics and its Applications LaMa-Lebanon},%Department and Organization
            addressline={Lebanese University},
            country={Lebanon}}

\begin{abstract}
The Symmetric Tensor Approximation problem (STA) consists of approximating a symmetric tensor or a homogeneous polynomial by a linear combination of symmetric rank-1 tensors or powers of linear forms of low symmetric rank. We present two new Riemannian Newton-type methods for low rank approximation of symmetric tensor with complex coefficients.

The first method uses the parametrization of the set of tensors of rank at most $r$ by weights and unit vectors.
Exploiting the properties of the apolar product on homogeneous polynomials combined with  efficient tools from complex optimization, we provide an explicit and tractable formulation of the Riemannian gradient and Hessian, leading to Newton iterations with local quadratic convergence. We prove that under some regularity conditions on non-defective tensors in the neighborhood of the initial point, the Newton iteration (completed with a trust-region scheme) is converging to a local minimum.

The second method is a Riemannian Gauss--Newton method on the Cartesian product of Veronese manifolds. An explicit orthonormal basis of the tangent space of this Riemannian manifold is described. We deduce the Riemannian gradient and the Gauss--Newton approximation of the Riemannian Hessian. We present a new retraction operator on the Veronese manifold.

We analyze the numerical behavior of these methods, with an initial point
provided by Simultaneous Matrix Diagonalisation (SMD).
Numerical  experiments  show  the good numerical behavior  of  the  two methods in different cases and in comparison with existing state-of-the-art methods.
\end{abstract}

\begin{keyword}
symmetric tensor decomposition\sep homogeneous polynomials\sep Riemannian optimization\sep Newton method\sep retraction\sep complex optimization\sep trust region method\sep Veronese manifold.

\MSC[] 15A69\sep 15A18\sep 53B20\sep 53B21\sep 14P10\sep 65K10\sep 65Y20\sep 90-08
\end{keyword}

\end{frontmatter}


\section{Introduction}
A symmetric tensor $T$ of order $d$ and dimension $n$ in $\T^d(\C^n)=\C^{n}\otimes\cdots\otimes\C^{n}= \T_{n}^{d}$ is a special case of tensors, where its entries do not change under any permutation of its $d$ indices. We denote their set by $\S^d(\C^n)=\ts$.
The symmetric tensor decomposition problem consists of decomposing a symmetric tensor $T \in \ts$ into a linear combination of symmetric tensors of rank one i.e.
\begin{equation}\label{e1}
    T=\sum_{i=1}^{r}{w_i \underbrace{{v_i\otimes...\otimes v_i}}_{d~\text{times}}},~~ w_i \in \C,~ v_i\in \C^n
\end{equation}
For a multilinear tensor, its minimal decomposition as a sum of tensor products of vectors is called the Canonical Polyadic Decomposition \cite{doi:10.1002/sapm192761164}. We have a correspondence between $\ts$ and the set of homogeneous polynomials of degree $d$ in $n$ variables denoted $\C[x_1,\dots,x_n]_d=:\C[\x]_d$. Using this correspondence, (\ref{e1}) is equivalent to express the homogeneous polynomial $\bm p$ associated to $T$ as a sum of powers of linear forms, which is by definition the classical Waring decomposition i.e.
\begin{equation}\label{e2}
    \bm p=\sum_{i=1}^{r}{w_i (v_{i,1}x_1+\dots+v_{i,n}x_n)^d},~~ w_i \in \C,~ v_i\in \C^n
\end{equation}
The smallest $r$ such that this decomposition exists is by definition the symmetric rank of $\bm p$ denoted by $\rank_s(\bm p)$. Let $d\ge3$. The generic symmetric rank denoted by $r_g$, is given by the Alexander--Hirschowitz theorem \cite{Alexander1995} as follows: $r_g=\big\lceil{\frac{1}{n}\binom{n+d-1}{d}}\big\rceil$ for all $n, d\in\mathbb{N}$, except for the following cases:
$(d,n)\in\left\{(3,5),(4,3),(4,4),(4,5)\right\}$, where it should be increased by $1$.
We say that $T$ is of subgeneric rank, if its rank $\rank_{s}(T)=r$ in (\ref{e2}) is strictly lower than $r_g$. In this case, a strong property of uniqueness of the Waring decomposition holds \cite{Chiantini_2016}, and the symmetric tensor $T$ is called identifiable, unless in three exceptions which are cited in \cite[Theorem 1.1]{Chiantini_2016}, where there are exactly two Waring decompositions. This identifiability property forms an important key strength of the Waring or equivalently the symmetric tensor decomposition. It can explain why this decomposition problem appears in many applications for instance in the areas of mobile communications, in blind identification of under-determined mixtures, machine learning, factor analysis of k-way arrays, statistics, biomedical engineering, psychometrics, and chemometrics. See e.g. \cite{comon:hal-00347139,COMON20062271,doi:10.1137/S0895479896305696,doi:10.1002/cem.908,khouja2021tensor} and references therein. The decomposition of the tensor is often used to recover structural information in the application problem.

The Symmetric Tensor Approximation problem (STA) consists of finding the closest symmetric tensor to a given symmetric tensor $T \in \ts$, of low symmetric rank. Equivalently, for a given $r\in \NN^*$, it consists of approximating a homogeneous polynomial $\bm p$ associated to a symmetric tensor $T$ by an element in $\Sigma_{r}$, where $\Sigma_{r}=\left\{\bm q\in\hp\mid \rank_s(\bm q)\le r\right\}$, i.e.
\begin{equation*}
    \text{(STA)}~~~\min_{\bm q\in\Sigma_{r}}\tfrac{1}{2}||\bm p-\bm q||_d^2.
\end{equation*}
Since in many problems, the input tensors are often computed from measurements or statistics, they are known with some errors on their coefficients and computing an approximate decomposition of low rank often gives better structural information than the exact or accurate decomposition of the approximate tensor \cite{allman2009,JMLR:v15:anandkumar14b,Garcia_2005}.

For matrices, the best low rank approximation can be computed via Singular Value Decomposition (SVD).
Higher Order Singular Value Decomposition (HOSVD) has been investigated to compute a multilinear rank approximation of a tensor
\cite{doi:10.1137/S0895479896305696, doi:10.1137/S0895479898346995, %Lathauwer1995HigherorderPM,
VannieuwenhovenNewTruncationStrategy2012}, this, in contrast to the matrix case, does not give the best multilinear rank approximation (see for instance inequality (5) in \cite{Kressner_SV_2013}).

A classical approach for computing an approximate tensor decomposition of low rank is the so-called Alternating Least Squares (ALS) method. It consists of minimizing the distance between a given tensor and a low rank tensor by alternately updating the different factors of the tensor decomposition, solving a quadratic minimization problem at each step. See e.g. \cite{Carroll1970, 3923712d8c7b4f8da616f88a35b3f34b,harshman70,KoBa09}. This approach is well-suited for tensor represented in $\td$ but it looses the symmetry property in the internal steps of the algorithm.
The space in which the linear operations are performed is of large dimension $n^d$ compared to the dimension $\genfrac(){0pt}{1}{n+d-1}{d}$ of $\ts$ when $n$ and $d$ grow. Moreover the convergence is slow \cite{EMH,doi:10.1137/110843587}.

Other iterative methods such as quasi-Newton methods have been considered for low rank tensor approximation problems to improve the convergence speed. See for instance \cite{Hayashi1982,doi:10.1080/10618600.1999.10474853,doi:10.1137/100808034,doi:10.1137/090763172,doi:10.1137/120868323,10.1016/j.csda.2004.11.013}. A Riemannian Gauss--Newton algorithm with trust region scheme was presented in \cite{Breiding2018}, to approximate a given real multilinear tensor by one of low rank. The Riemannian optimization set is a Cartesian product of Segre manifolds (i.e. manifolds of real multilinear tensors of rank one).
The retraction on the Segre manifold, called ST-HOSVD, is based on sequentially truncated HOSVD \cite{Kressner_SV_2013,hackbusch2012tensor,VannieuwenhovenNewTruncationStrategy2012}. Moreover, an algorithm, called hot restarts, was introduced in \cite{Breiding2018} to avoid ill-conditioned decompositions. 
Closely related to these iterative methods, the condition number of join decompositions such as tensor decompositions is studied in \cite{Breiding2018b}.

Optimization techniques based on quasi-Newton iterations for block term decompositions of multilinear tensors over the complex numbers have also been presented in \cite{Sorber2012,doi:10.1137/120868323}. In \cite{doi:10.1137/090763172} quasi-Newton and limited memory quasi-Newton methods for distance optimization on products of Grassmannian manifolds are designed to deal with the Tucker decomposition of a tensor and applied for a low multilinear rank tensor approximation. In all these approaches, an approximation of the Hessian is used to compute the descent direction, and the local quadratic convergence cannot be guaranteed.

Specific investigations have been developed, in the case of best rank-$1$ approximation. The problem is equivalent to the optimization of a polynomial on the product of unitary spheres (see e.g. \cite{doi:10.1137/S0895479898346995,doi:10.1137/110835335}). Global polynomial optimization methods can be employed over the real or complex numbers, using for instance convex relaxations and semidefinite programming \cite{doi:10.1137/130935112}. However, the approach is facing scalability issues in practice for large size tensors.

In relation with polynomial representation and multivariate Hankel matrix properties, another least square optimization problem is presented in \cite{NieLowRankSymmetric2017}, for low rank symmetric tensor approximation. Good approximations of the low rank approximation are obtained for small enough perturbations of low rank tensors. More recently, a method for decomposing real even-order symmetric tensors, called Subspace Power Method (SPM), has been proposed in \cite{KileelSubspacepowermethod2019}. It is based on a power method associated to the projection on subspaces of eigenvectors of the Hankel operators and has a linear convergence.
\subsection{Contributions}In this paper, we present two new Riemannian Newton-type methods for the low rank approximation problem (STA) for symmetric tensors with complex coefficients.

The first method uses the parametrization of the set of tensors of rank at most $r$ by weights and vectors on the unit sphere.
Exploiting the properties of the apolar product on homogeneous polynomials combined with  efficient tools from complex optimization, we provide an explicit and tractable formulation of the Riemannian gradient and Hessian, leading to Newton iterations with local quadratic
convergence. We prove that under some regularity conditions on non-defective tensors in the neighborhood of the initial point, the iteration (completed with a trust-region scheme) is converging to a local minimum.

The second method is a Riemannian Gauss--Newton method on the Cartesian product of Veronese manifolds.
We describe an explicit orthonormal basis of the tangent space of this Riemannian manifold. We use this basis to obtain the Riemannian gradient and the Gauss--Newton approximation of the Riemannian Hessian. We present an approximation method for a given homogeneous polynomial in $\hp$ into linear form to the $d\textsuperscript{$th$}$ power, based on the rank-1 truncation of the SVD of Hankel matrix associated to the homogeneous polynomial. From this approximation method, we propose a new retraction operator on the Veronese manifold. We point out that the Riemannian Gauss--Newton iteration on the Cartesian product of Segre manifolds presented in \cite{Breiding2018} is related to our approach in the sense that the design of the algorithm depends mainly on the geometry of the targeted manifold (Segre manifold in the multilinear case, Veronese manifold in the symmetric case) and its tangent space. In this context, the Riemannian Gauss--Newton iteration that we describe is adapted to the symmetric setting by considering the reduced vector space $\C[\x]_d$ and by exploiting the apolar identities. In our approach we consider symmetric tensors with complex coefficients. The constraint set is parameterized via the complex Veronese manifolds, which leads us to a complex optimization problem with geometric constraints, and this, to the best of our knowledge, has not been addressed previously in tensor approximation. 

We analyze the numerical behavior of these methods, choosing for the initial point
the approximate decomposition provided by the Simultaneous Matrix Diagonalisation (SMD) of a pencil of Hankel matrices \cite{harmouch:hal-01440063,mourrain:hal-01367730}.
Numerical  experiments  show  the good numerical behavior  of  the  new  methods for the best rank-1  approximation  of real-valued  symmetric tensors, for low rank approximation of sparse symmetric tensors, and against  perturbations of symmetric tensors of low rank. Comparisons with existing state-of-the-art methods corroborate this analysis.
\subsection{Outline}
The paper is structured as follows. In \cref{Notation and preliminaries} we give the main notation and preliminaries. In \cref{non-defective}, we describe the set of non-defective rank-$r$ symmetric tensors.
In subsection \ref{formulation}, we formulate the STA problem as a Riemannian least square optimization problem using the parametrization by weights and unit vectors.
We compute explicitly the Riemannian gradient vector and the Hessian matrix in subsection \ref{computation} (the proofs are in \ref{append}) and describe the retraction in subsection \ref{retnr}.
In subsection \ref{ver}, we describe the Riemannian Gauss--Newton method on the product of Veronese manifolds. We present in subsection \ref{retraction1} a new retraction operator on the Veronese manifold with its analysis.
In subsection \ref{trust}, we recall the trust-region extension scheme, and prove under some regularity assumptions the convergence of the exact Riemannian Newton method with trust region steps to a local minimum of the distance function.
Numerical experiments are featured in \cref{num}. The final section is for our conclusions and outlook.

\section{Notation and preliminaries}\label{Notation and preliminaries} We use similar notation as in \cite{Comon2008}.
We denote by $\td=\T^d(\C^n)=\C^{n}\otimes\cdots\otimes\C^{n}$ the outer product $d$ times of $\C^n$. The set of symmetric tensors in $\td$ is denoted $\ts$.
We have a correspondence between $\ts$ and the set of the homogeneous polynomials of degree $d$ in $n$ variables $\C[x_1,\dots,x_n]_d:=\C[\x]_d$.
%To every symmetric tensor $\bm P$ of order $d$ and dimension $n$, we  uniquely associated a homogeneous polynomial of degree $d$ in $n$ variables in $\C[\x]_d$,
This allows to reduce the dimension of the ambient space of the problem from $n^d$ (dimension of $\td$) to $s_{n,d}:=\binom{n+d-1}{d}$ (dimension of $\ts \sim \C[\x]_d$).
The bold letters $\bm p, \bm q$ denote homogeneous polynomials in $\hp$ or equivalently elements in $\ts$.
%The multilinear tensor associated to the symmetric tensor or homogeneous polynomials $\bm p$ is denoted $\bm p^\tau\in \td$.
A homogeneous polynomial $\bm p$ in \hp~can be written as: $\bm p=\sum_{|\alpha|=d}{\binom{d}{\alpha}p_{\alpha}\x^{\alpha}}$, where $\x:=(x_1, \dots, x_n)$ is the vector of the variables $x_1, \dots, x_n$, $\alpha=(\alpha_1, \dots, \alpha_n)$ is a vector of the multi-indices in $\NN^n$, $\vert \alpha \vert = \alpha_{1}+ \cdots + \alpha_{n}$, $p_\alpha\in\C$, $\x^\alpha:=x_1^{\alpha_1}\dots x_n^{\alpha_n}$ and $\binom{d}{\alpha}:=\frac{d!}{\alpha_1!\dots\alpha_n!}$.
%Thus we can write the multilinear tensor $\bm p^\tau$ as $\bm p^\tau=\sum_{1\le i_1,\dots,i_d\le n}(\sum_{\alpha\mid e_{i_1}+\dots+e_{i_d}=\sum_{i=1}^{n}{\alpha_i e_i} }{p_\alpha})\,e_{i_1}\otimes\dots\otimes e_{i_d}$, where $(e_i)_{1\le i\le n}$ denotes the canonical basis of $\C^n$.
The superscripts $.^t$, $.^*$ and $.^{-1}$ are used respectively for the transpose, Hermitian conjugate, and the inverse matrix. Let $A\in\C^{n\times n}$, we denote by $\sqrt{A}$ a matrix $B\in\C^{n \times n}$ such that $A=B^2$. The complex conjugate is denoted by an overbar, e.g., $\bar{w}$. We use parentheses to denote vectors e.g. $W=(w_i)_{1\le i\le r}$, and the square brackets to denote matrices e.g. $V=[v_i]_{1\le i\le r}$ where $v_i$ are column vectors. The concatenation of vectors $v_{1}, v_{2}, \ldots$ is denoted $(v_{1};v_{2}; \ldots)$.
%Let $V_{vec}:=\begin{pmatrix}v_1\\\vdots\\v_n\end{pmatrix}$.
\begin{definition}\label{norm}
  For   $\bm p=\sum_{|\alpha|=d}{\binom{d}{\alpha}p_{\alpha}\x^{\alpha}}$ and $\bm q=$ $\sum_{|\alpha|=d}{\binom{d}{\alpha}q_{\alpha}\x^{\alpha}}$ in $\C[\x]_d$,  their apolar product is
$$
\langle \bm p,\bm q\rangle_d:=\sum_{|\alpha|=d}{\binom{d}{\alpha}\bar{p}_{\alpha}q_{\alpha}}.
$$
The apolar norm of $\bm p$ is $||\bm p||_d=\sqrt{\langle \bm p,\bm p\rangle_d}=\sqrt{\sum_{|\alpha|=d}{\binom{d}{\alpha}\bar{p}_{\alpha}p_{\alpha}}}$.
\end{definition}
% For $T\in \td$, $vec(T)$ denotes the vectorization of $T$ i.e. $vec(T)\in\C^{n^d}$.
% We define the Frobenius norm of $T$ by: $||T||_F=||vec(T)||_{2}= \sqrt{vec(T)^*vec(T)}$.
% From the definition of $\bm p^{\tau}$, we have the following property:
% \begin{lemma}
%   For $\bm p\in \ts$, we have $||\bm p||_d=||\bm p^\tau||_F$.
% \end{lemma}
The following properties of the apolar product can be verified by direct calculus:
\begin{lemma}\label{ap}
Let $\bm l=(v_1x_1+\cdots+v_nx_n)^d:=(v^t\x)^d \in \hp$ where $v=(v_i)_{1\le i\le n}$ is a vector in $\C^n$, $\bm q\in\C[\x]_{(d-1)}$, we have the following two properties:
\begin{enumerate}
    \item $\langle \bm l,\bm p\rangle_d=\bm p(\bar{v})$, $\langle \bm p, \bm l\rangle_d=\bar{\bm p}(v)$,
    \item $\langle \bm p,x_i\bm q\rangle_d=\frac{1}{d}\langle \partial_{x_i}\bm p,\bm q\rangle_{d-1}, \langle x_i\bm q, \bm p\rangle_d=\frac{1}{d}\langle \bm q,\partial_{x_i} \bm p\rangle_{d-1},~~\forall 1\le i\le n$.
\end{enumerate}
\end{lemma}
\section{The set of non-defective rank-$r$ symmetric tensors}\label{non-defective}

Let $\Sigma_{r}\subset \ts$ be the set of symmetric tensors of symmetric rank at most $r$.
A symmetric tensor $\bm t\in \Sigma_r$ is the sum of $d^{\textup{th}}$ powers
\begin{equation}\label{decomp}
\bm t = \sum_{i=1}^{r} (v_{i}^{t} \x)^{d},~\text{for}~v_{i} \in \C^{n}.
\end{equation}
 It is a point in the image of the following map:
\begin{align*}
 \psi_{r}:&~\C^{n \times r}:=\C^n\times\cdots\times\C^{n}\longrightarrow\hp\\
 &[v_i]_{1\le i\le r}\longmapsto\psi_r((v_i)_{1\le i\le r})=\sum_{i=1}^{r}{(v_i^t\x)^d}.
\end{align*}
The $d^{\textup{th}}$ power $(v_{i}^{t} \x)^{d}$ with $v_i\neq 0$ are symmetric tensors of rank-$1$, which are on the so-called Veronese manifold.
\begin{definition}
Let $\psi:\C^n\to\hp,~v\mapsto(v^t\x)^d=\sum_{|\alpha|=d}{\binom{d}{\alpha}v^\alpha \x^\alpha}$.
The Veronese manifold in $\hp$ denoted by $\V^{n,d}$ is the set of linear forms in $\C[\x]_{1}-\left\{0\right\}$ to the $d$\textsuperscript{$th$} power. It is  the image of $\C^n-\left\{0\right\}$ by $\psi$.
\end{definition}
The Veronese variety studied in algebraic geometry is the algebraic variety of the projective space $\PP^{s_{n,d}-1}$ associated to $\Vnd$, where $s_{n,d}=\text{dim}~ \C[\x]_d$ \cite{AG, secant,landsberg2011tensors}. The tangent space of $\Vnd$ at a point $p=(v^{t}\x)^{d}$ is the vector space spanned by $\langle x_{1} (v^{t}\x)^{d-1}, \ldots$, $x_{n} (v^{t}\x)^{d-1} \rangle$, that is the linear space $T_{p}(\Vnd)=\{ (u^{t} \x) (v^{t } \x)^{d-1}\mid u\in \C^{n}\}$.

The Zariski closure $\sec_{r}$ of $\Sigma_{r}$ is called the $(r-1)${\em th-secant variety} of the Veronese variety. For $r>1$, the algebraic variety $\sec_{r}$ is not smooth and contrarily to the
case of matrices, singular points of $\sec_{r}$ can have a rank $>r$, as shown in the following example.
For $d>2$, $\bm p=(v_0^t\x)(v_1^t\x)^{d-1}\in \hp$ with $v_0 \neq v_1\in\C^n$ is in the (Zariski) closure of $\sec_{2}$ since
$(v_0^t\x)(v_1^t\x)^{d-1}= \lim_{\delta \rightarrow 0}\frac{1}{d\, \delta}(((v_1+\delta v_0)^{t} \x)^{d} - (v_{1}^{t}\x)^{d})$ but its symmetric rank is $d>2$ \cite[Proposition 5.6]{Comon2008}.

To avoid these singularities, we will restrict our theoretical analysis to points of $\Sigma_{r}$ where the map $\psi_{r}$ is a local embedding, since in the vicinity of singularities, the best low rank approximation problem is ill-posed (as shown by the previous example).
The map $\psi_{r}$ is a local embedding at $y = [v_{i}]_{1\le i\le r}\in {\C}^{n \times r}$ iff
$$
J\psi_{r}(y) = d~[x_{1} (v_{1}^{t}\x)^{d-1}, \ldots, x_{n} (v_{1}^{t}\x)^{d-1}, \ldots, x_{1} (v_{r}^{t}\x)^{d-1}, \ldots, x_{n} (v_{r}^{t}\x)^{d-1}]
$$
is of rank $n\, r$. The tensors $\psi_{r}(y)$ with  $y \in {\C}^{n \times r}$ such that $\rank\, J\psi_{r}(y)= n\, r$ are called {\em non-defective}.
The set of non-defective tensors of rank $r$, locally embedded in $\C[\x]_d$, is the image by a local diffeomorphism of a Riemannian manifold and it is denoted $\Sigma_r^{\mathrm{reg}}$.
The map $\psi_{r}$ is a local diffeomorphism between an open subset of $\C^{n \times r}$ and $\Sigma_{r}^{\mathrm{reg}} \subset \sec_{r}$.

Hereafter, we consider the cases where $d>2$ and the rank $r$ is strictly subgeneric, i.e. $r<r_g=\big\lceil\frac{1}{n}\binom{n+d-1}{d}\big\rceil$, where $r_g$ is the generic symmetric rank (except for $(d,n)\in \{(3,5), (4,3), (4,4), ($\linebreak$4,5) \}$ or $d=2$) by Alexander--Hirschowitz theorem \cite{Alexander1995}.
Using “Terracini's lemma” (see e.g. \cite[Lemma 5.3.1.1]{landsberg2011tensors}), we have that $\Sigma_r^{\mathrm{reg}}$ is a dense open subset of $\sec_{r}$ iff the dimension of $\sec_{r}$ is the expected dimension $n\,r$. In this case, $\sec_{r}$ is also said to be {\em  non-defective}. Alexander and Hirschowitz \cite{Alexander1995} proved that $\sec_{r}$ is non-defective when $r<r_{g}$ (the exceptional defective cases for $d>2$ being $(d,n,r)\in \{(3,5,7), (4,3,5), (4,4,9), (4,5,14) \})$.

It is also known that for $r<r_{g}$, generic tensors of $\psi_{r}$ have a unique decomposition, i.e. a unique inverse image by $\psi_{r}$ up to permutations, except for $(d,n,r)\in\{(6,2,9),(4,3,8),(3,5,9)\}$, see \cite[Theorem 1.1]{Chiantini_2016}.
\section{Riemannian optimization for the STA problem}
In this section, we use the framework of Riemannian optimization \cite{AbsMahSep2008} to solve the STA problem. See also \cite{Breiding2018,hackbusch2012tensor,Kressnerl_SV_2013} for real multilinear tensors. We develop a Riemannian Newton algorithm and a Riemannian Gauss--Newton algorithm exploiting the properties of symmetric tensors to obtain explicit and simplified formulation.
We consider distance minimization problems for symmetric tensors with complex decompositions for both algorithms.\\
Riemannian optimization methods are solving optimization problems over a Riemannian manifold $\M$ \cite{AbsMahSep2008}.
Given $\bm p\in\ts\sim\hp$, we consider hereafter the following least square minimization problem
\begin{equation}\label{leasteq}
\min_{y\in\M}f(y)
\end{equation}
where $f:\M\rightarrow\R$ is half the square distance function to ${\bm p}$ i.e. $f(y)=\frac{1}{2}\|F(y)\|_d^2$ with $F(y)=\Phi_r(y)-\bm p$, such that $\Phi_{r}:~\M\rightarrow\hp$ is a parametrization map of $\Sigma_{r}$ the set of symmetric tensors of symmetric rank bounded by $r$, and $\M$ is a Riemannian manifold. A Riemannian optimization method for solving (\ref{leasteq}) requires a Riemannian metric. Since we will assume that $\M$ is embedded in some space $\R^M$, we will take the metric induced by the Euclidean space $\R^{M}$.

% A Riemannian optimisation method for solving \cref{leasteq} requires a Riemannian metric. Since we will assume that $\M$ is embedded in some space $\R^M$, we will take the metric induced by the Euclidean space $\R^{M}$.

We propose to parametrize $\Sigma_{r}$, first via weights and unit vectors. We describe an exact Riemannian Newton method for this formulation in subsection \ref{formulation}. Secondly, we parametrize $\Sigma_{r}$ via sums of the $d$\textsuperscript{$th$} power of linear forms that is as sums of tensors in ${\Vnd}$. We develop a Riemannian Gauss--Newton method for this formulation in subsection \ref{ver}. A dogleg trust-region scheme in subsection \ref{trust} is added to the two algorithms.

Recall that a Riemannian Newton method for solving (\ref{leasteq}) \cite[Chapter 6]{AbsMahSep2008} consists of starting with an initial guess $y_0\in\M$ and generating a sequence $y_1,y_2, \dots$ in $\M$, with respect to the following process:\begin{equation}\label{newtoneq}
y_{k+1}\leftarrow R_{y_k}(\eta_k)\hspace{0.3cm}\text{with} \Hess f(y_k)[\eta_k]=-\grad f(y_k);\end{equation}
where $\grad f(y_k)$ and $\Hess f(y_k)$ are respectively the Riemannian gradient and Hessian of $f$ at $y_k$ on $\M$, and $R_{y_{k}}:T_{y_k}\M\to\M$ is a retraction operator from the tangent space $T_{y_{k}}\M$ to $\M$.

A Riemannian Gauss--Newton method \cite[Chapter 8]{AbsMahSep2008} is a Riemannian quasi-Newton method where the Riemannian Hessian in (\ref{newtoneq}) is replaced by $(DF(y_k))^*\circ(DF(y_k))$, namely the Gauss--Newton approximation of the Hessian.

The properties of a retraction map are described hereafter:
\begin{definition}\cite{AbsMahSep2008, 10.1093/imanum/22.3.359,Kressnerl_SV_2013}\label{conditions}
  Let $\M$ be a manifold and $y\in \M$. A retraction $R_y$ is a map $T_y\M\rightarrow\M$, which satisfies the following properties :
\begin{enumerate}
    \item $R_y(0_y)=y$;
    \item there exists an open neighborhood $\U_y\subset T_y\M$ of~$0_y$ such that the restriction on $\U_y$ is well-defined and a smooth map;
    \item $R_y$ satisfies the local rigidity condition\begin{equation*}
        DR_y(0_y)=id_{T_y\M},
    \end{equation*}
    where $id_{T_y\M}$ denotes the identity map on $T_y\M$.
\end{enumerate}
\end{definition}
We will also use the following property which is straightforward to show:
\begin{lemma}\label{decomposeret}
Let $\M_1,\ldots,\M_r$ be manifolds, $y_i\in\M_i$ and $\M=\M_1\times\cdots\times\M_r$ and $y=(y_1, \ldots, y_r)\in \M$. Let $R_i: T_{y_i}\M_i\rightarrow\M_i$ be retractions. Then $R_y:T_y\M \rightarrow \M$ defined as follows:
$R_{y}(\xi_{{1}},\ldots,\xi_{{r}})=(R_{y_{1}}(\xi_{{1}}),\ldots,R_{y_{r}}(\xi_{{r}}))$ for $\xi_{i}\in T_{y_i}\M_i$, $1\le i\le r$, is a retraction on $\M$.
\end{lemma}

\subsection{Riemannian Newton method for STA}\label{formulation}

% We describe now the Riemannian manifold that we use for the Riemannian Newton method.
% Using the local diffeomorphism $\Sigma_{r}$ between the Riemannian manifold $\M_r$ and $\sigma_{r}^{reg}$, we reformulate the Riemannian least square problem for the symmetric tensor  approximation problem as follows:
% \begin{equation*}
%     (STA^*) ~~\min_{y\in\M_r}f(y)
% \end{equation*}
% where $\M_r=\{ y\in {\Vnd}^{\times r}\mid \rank J\Sigma_{r}(y) = n\,r\}$, $y=((v_i^t\x)^d)_{1\le i\le r}\in \M_r$, $f(y)=\frac{1}{2}||F(y)||_d^2$ and $F(y)=\Sigma_r(y)-\bm p.$

% \subsection{Parameterization}\label{para}

We normalize the decomposition (\ref{decomp}) by choosing unit vectors for $v_i$ and positive weights. Namely, we decompose a symmetric tensor $\bm p\in\Sigma_r$ as $\bm p = \sum_{i=1}^{r} w_{i}\, (v_i^t\x)^d$
with $w_{i} \in \R_{+}^{*}$ and $||v_i||=1$, for $1\le i\le r$; by normalizing $v_i$ and multiplying by $w_i:=||v_i||^d$ if $v_i$ is not a unit vector. The vector $(w_i)_{1\le i\le r}$ in this decomposition is called ``the weight vector", and is denoted by $W$. Let $V=[v_i]_{1\le i\le r}\in\C^{n\times r}$
be the matrix of the normalized vectors.

The objective function expressed in terms of these weights and unit vectors is given by $f(W,V)=\tfrac{1}{2}||F(W,V)||_d^2, \text{with}~F(W,V)=\sum_{i=1}^{r} w_{i}\, (v_i^t\x)^d-\bm p$.

The function $f$ is a real valued function of complex variables; such function is  non-analytic, because it cannot satisfy the Cauchy--Riemann conditions \cite{remmert1991theory}.
To apply the Riemannian Newton method, we need the second order differentials of $f$.
As discussed in \cite{Sorber2012}, we overcome the non-analytic problem by converting the optimization problem to the real domain, regarding $f$ as a function of the real and imaginary parts of its complex variables.\\

% Note that (STA) is not a Riemannian least square problem since the constraint set $\psi_r$ is not a Riemannian manifold\footnote{It is a semi-algebraic set \cite{de_Silva_2008, real}.}.

Let $\N_{r}=\{(W,\Re(V),\Im(V))\mid W\in{\R^*_+}^r, V\in\C^{n\times r}, (\Re(v_i),\Im(v_i))\in\SS^{2n-1}$, $\forall\, 1\le i\le r\}$, where $\SS^{2n-1}$ is the unit sphere in $\RR^{2n}$. Let ${\varphi}_{r}: (w,v_{1}, \ldots, v_{r}, v'_{1}, \ldots, v'_{r} )\in \N_{r} \mapsto \sum_{i=1}^{r} w_{i} ((v_{i}+\ib\, v'_{i})^{t} \x)^{d}$. Hereafter in this subsection, we use the following formulation to compute the different ingredients of a Riemannian Newton method:
\begin{equation*}
    \text{(STA)}_{\N_r}~~~\min_{y\in\N_r} f(y),
\end{equation*}
where $f(y)=\tfrac{1}{2}||F(y)||_d^2, \text{with}~F(y)=\varphi_r(y)-\bm p$.\\

\subsubsection{Computation of the gradient vector and the Hessian matrix}\label{computation}In this section, we present the explicit expressions of the Riemannian gradient and Hessian on $\N_{r}$.
We first describe an orthonormal basis of $T_{y}\N_{r}$ for $y\in \N_{r}$. Then we detail the computation of the gradient and Hessian in this basis, via the differentials of maps in complex and conjugate variables.

\begin{lemma}\label{base}
  Let $y=(w, v_{1}, \ldots, v_{r}, v'_{1},\ldots, v'_{r})\in\N_{r}$. For all $i=1, \ldots, r$ let $\dbl{v}_i=(v_i; v'_{i})\in \SS^{2n-1}$ and let
$$
(I_{2n}-\dbl{v}_i\dbl{v}_i^t)=Q_{i}R_{i} P_{i}
$$
be a rank-revealing QR-decomposition of the projector on $\dbl{v}_{i}^{\perp}$ in $\R^{2n}$,
where $Q_{i}Q_{i}^t=I_{2n}$, $R_{i}$ is upper triangular, and $P_{i}$ is a permutation matrix.

Let $Q_{i,re}$ (resp. $Q_{i,im}$) be the matrix given by the first $n$ rows (resp. the last $n$ rows)
    and the first $2n-1$ columns of $Q_{i}$.
Let $\tilde{Q}=\begin{bmatrix}Q_{re}\\ Q_{im}\end{bmatrix}\in\R^{2nr\times (2n-1)r}$, where $Q_{re}=\mathrm{diag}(Q_{i,re})_{1\le i\le r}$ and $Q_{im}=\mathrm{diag}(Q_{i,im})_{1\le i\le r}$. Then the columns of~~$\B=\mathrm{diag}(I_r,\tilde{Q})$ form an orthonormal basis of $T_y\N_{r}$.
\end{lemma}
\begin{proof}
  We have $T_y\N_{r}\simeq T_w(\R_+^*)^{r}\times T_Z\S_r$, where $\S_r=\{(\Re(V),\Im(V))\mid V\in\C^{n\times r}$,\\$||v_i||^2=1, ~\forall 1\le i\le r\}$ and $Z={(\Re(V),\Im(V))}= (v_{1}, \ldots, v_{r}, v'_{1}, \ldots, v'_{r})\in \R^{n \times 2 r}$.

As $T_w(\R_+^*)^{r}=\R^{r}$, $I_r$ represents an orthonormal basis of $T_w(\R_+^*)^{r}$.

We verify now that $\tilde{Q}$ is an orthonormal basis of $T_Z\S_r$.
For $i=1, \ldots,r$, $\dbl{v}_i\in\SS^{2n-1}\subset\R^{2n}$ and the first $(2n-1)$ columns of the factor $Q_{i}$ of a rank-revealing QR-decomposition of $I_{2n}-\dbl{v}_i\dbl{v}_i^t$ give an orthonormal basis of the image $\dbl{v}_{i}^{\perp}$ of $(I_{2n}-\dbl{v}_i\dbl{v}_i^t)$,~ that is $T_{\dbl{v}_i}\SS^{2n-1}$.

The vector space $T_Z\S_r$, of dimension $r(2n-1)$, is the Cartesian product of the tangent spaces $T_{\dbl{v}_i}\SS^{2n-1}$. Therefore, by construction, the columns of $\tilde{Q}$ form an orthonormal basis of $T_Z\S_r$.

We deduce that $Q=\mathrm{diag}(I_r,\tilde{Q})$ represents an orthonormal basis of $T_y\N_{r}$ in the canonical basis of $\R^{2n r}$.
\end{proof}

Let $\cR_r=\left\{(W,\Re(V),\Im(V))\in \R^{r}\times \R^{n\times r} \times \R^{n\times r} \mid W\in\R^r, V\in\C^{n\times r}\right\}$ and  let $f_R$ be the function $f$ seen as a function on $\cR_{r}$.
The gradient and the Hessian of $f_R$ at a point $p^R \in \cR_{r}$  are called the real gradient and the real Hessian. We denote them by $G^R$ and $H^R$. We will describe their computation, after the next proposition, relating them to the Riemannian gradient and Hessian.
\begin{proposition}\label{prop:newton:nr}
Let $p=(w, v_1, \dots, v_r, v\sp{\prime}_1, \dots, v\sp{\prime}_r )\in\N_r$, $\B\in\R^{(r+2nr)\times( r+(2n-1)r)}$ such that its columns form an orthonormal basis of $T_{y}\N_{r}$. Let $G^R=(g_0, g_1, \dots, g_r, g\sp{\prime}_1,$$ \dots, g\sp{\prime}_r)\in\R^{r+2nr}$ (resp. $H^R\in\R^{(r+2nr)\times (r+2nr)}$) be the gradient vector (resp. the Hessian matrix) of $f_R$ at $p^R$ in the canonical basis. The Riemannian gradient vector (resp. Hessian matrix) of $f$ at $p$ with respect to the basis $Q$ is given by:
$$
G=\B^tG^R, H=\B^t(H^R+S)\B,
$$
where $S=\mathrm{diag}(0_{r\times r},\tilde{S},\tilde{S})$, with $\tilde{S}=\mathrm{diag}(s_1I_n, \dots, s_rI_n)$, $s_i=\langle v_i, g_i\rangle+\langle v\sp{\prime}_i, g\sp{\prime}_i\rangle$.
\end{proposition}
The proof is given in \ref{proof1}.\\
Let us describe now explicitly the real gradient $G^{R}$:
\begin{proposition}\label{gradient}
The gradient $G^R$ of $f_R$ on $\mathcal{R}_{r}$ is the vector
$$
G^{R}= \begin{pmatrix}G_1\\\mathrm{\Re}(G_2)\\-\Im(G_2)\end{pmatrix}\in\R^{r+2nr},
$$
where
\begin{itemize}
\item $G_1=(\sum_{i=1}^{r}{w_i\Re((v_j^*v_i)^d)}-\Re(\bar{\bm p}(v_j)))_{1\le j\le r}\in\R^r$, \item $G_2=(d\sum_{i=1}^{r}{w_iw_j({v}_i^*v_j)^{(d-1)}\bar{v}_i}-w_j\nabla_\x\bar{\bm p}(v_j))_{1\le j\le r}\in\C^{nr}$.
\end{itemize}
\end{proposition}

The matrix of the real Hessian can be computed as follows:
\begin{proposition}\label{hessian}
The real Hessian matrix $H^R$ is the following block matrix:
  \begin{equation*}
      H^R=\begin{bmatrix}A&{\Re(B)}^t&-{\Im(B)}^t\\\Re(B)&\Re(C+D)&-\Im(C+D)\\-\Im(B)&\Im(D-C)&\Re(D-C)\end{bmatrix}\in\R^{(r+2nr)\times(r+2nr)},
  \end{equation*}
with
\begin{itemize}
 \item  $A=\Re([({v}_i^*v_j)^d]_{1\le i,j\le r})\in\R^{r\times r}$,
 \item $B=[d w_i({v}_j^*v_i)^{d-1}\bar{v}_j+\delta_{i, j}(d\sum_                                     {l=1}^{r}{w_l({v}_l^*v_i)^{d-1}\bar{v}_l}-\nabla_{\x} \bar{\bm p}(v_j))]_{1\le i, j\le r}\in \C^{nr\times r}$, \text{where} $\delta_{i, j}$ \text{is the Kronecker delta},
         \item $C=\mathrm{diag}[d(d-1)[\sum_{i=1}^{r}{w_iw_j\overline{v_{i,k}}\overline{v_{i,l}}({v}_i^*v_j)^{d-2}}]_{1\le k,l\le n}-w_j\Delta_\x\bar{\bm p}(v_j)]_{1\le j\le r}\in\C^{nr\times nr}$, where $\Delta_\x\bar{\bm p}(v_j):=[\partial_{x_k}\partial_{x_l}\bar{\bm p}(v_j)]_{1\le k, l\le n}$,
         \item $D=[dw_iw_j({v}_i^*v_j)^{d-2}(({v}_i^*v_j)I_n+(d-1)v_jv_i^*)]_{1\le i,j\le r}\in\C^{nr\times nr}.$
\end{itemize}
\end{proposition}
The \ref{proofs} is devoted to discussing the computation details of the real gradient and Hessian, where the proofs of propositions \ref{gradient} and \ref{hessian} are covered respectively in \ref{proof2} and \ref{proof3}.
\subsubsection{Retraction on $\N_r$}\label{retnr}
To complete this Riemannian Newton method, we need to define a retraction operator on $\N_r$. Let us assume that the Riemannian Newton equation is solved at a point
$y= (w, v_{1}, \ldots, v_{r}, v'_{1}, \ldots$\linebreak$v'_{r})\in\N_{r}$, in local coordinates with respect to the basis $\B$ as in lemma \ref{base}. It yields a solution vector $\hat{\eta}\in\R^{r+r(2n-1)}$. The tangent vector $\eta\in T_y\N_{r}$ of size $r+2nr$ is given by $\eta=\B\,\hat{\eta}= (\nu,\eta_{1},\ldots,\eta_{r}, \eta'_{1}, \ldots, \eta'_{r})$.
The new point $R_y(\eta)=(\tilde{w},\tilde{v}_{1}, \ldots,\tilde{v}_{r}, \tilde{v}'_{1}, \ldots, \tilde{v}'_{r})\in\N_{r}$ is defined using the product of the retractions on each component, that is the identity map on $\R^{r}$ and the projection map on the sphere $\SS^{2n-1}$  \cite{doi:10.1137/100802529} as follows:
\begin{itemize}
\item $\tilde{w}=R_{w}(\nu)=w+\nu$;
\item $(\tilde{v}_j,\tilde{v}'_{j}) = R_{(v_{j};v'_{j})}(\eta_{j},\eta'_{j})
  =  \frac{(v_j+\eta_j;v'_j+\eta'_{j})}{||(v_j+\eta_j;v'_j+\eta'_{j})||}$.
\end{itemize}
By lemma \ref{decomposeret}, this defines a retraction from $T_y\N_r$ to $\N_r$ since $R_{w}$ (resp. $R_{(v_{j};v'_{j})}$) is a retraction on $\R^{r}$ (resp. $\SS^{2n-1}$).

%In the case $\tilde{w}_j<0$, we take $\tilde{w}_j=-\tilde{w}_j$ and $\tilde{v}_j=e^{\mathbf{i}\frac{\pi}{d}}\tilde{v}_j$.\\

\subsection{Riemannian Gauss--Newton for STA}\label{ver}

In this subsection, we consider the STA problem over the product of $r$ Veronese manifolds $\Vnd$. By separating the real and imaginary parts of the coefficients of a polynomial, the non-zero points $(v^{t} \x)^{d}$ with $v\in \C^{n}\setminus \{0\}$ form a smooth Riemannian variety in $\C[\x]_{d}$.
We equip the $\R$-vector space $\C[\x]_{d} \sim \R^{2 s_{n,d}}$ with the real inner product:
$$
\forall \bm p, \bm q \in \C[\x]_{d}, \quad \apolar{\bm p, \bm q}_{d}^{\R} = \Re(\apolar{\bm p, \bm q}_{d}).
$$

Let $\V_r:=\Vnd\times\dots\times\Vnd$.
The map $\sigma_r:y=(y_1,\dots,y_r) \in \V_r \mapsto y_1+\dots+y_r \in \C[\x]_d$ is a parameterization  of the set $\Sigma_{r}$ of symmetric tensors of symmetric rank at most $r$.
We formulate the STA problem as a Riemannian least square problem over $\V_r$ as follows:
\begin{equation*}
    \text{(STA)}_{\V_r}~~~\min_{y\in\V_r} f(y),
\end{equation*}
where $f(y)=\tfrac{1}{2}||F(y)||_d^2, \text{with}~F(y)=\sigma_r(y)- \bm p$ for $y\in \V_{r}$.

The differential map $DF= D\sigma_r$ at $y= (y_{1}, \ldots, y_{r})\in \V_{r}$ with $y_i=(v_i^t\x)^d$, $v_{i}\in \C^{n}$ is
\begin{eqnarray*}
D\sigma_r(y): T_{y_1}\Vnd\times\dots\times T_{y_r}\Vnd
& \to & T_{\sigma_r(y)}\C[\x]_d=\C[\x]_d\\
(\eta_1,\dots,\eta_r)
&\mapsto& \eta_1+\dots+\eta_r,
\end{eqnarray*}
where $T_{y_i}\Vnd=\left\{(u^t\x)(v_i^t\x)^{d-1}|~u\in\C^n\right\}$ is of dimension $2n$ over $\R$.

Recall that the Gauss--Newton equation is given by \cite[Chapter 8]{AbsMahSep2008}:
\begin{equation}\label{gnv}
    (DF(y))^*\circ(DF(y))[\eta]=-(DF(y))^*[F(y)],
\end{equation}
where $DF(y):T_y\V_r\to\C[\x]_d$ is the differential map of $F$ at $y$. The map $(DF(y))^*\circ(DF(y)):T_y\V_r\to T_y\V_r$ is the so-called  Gauss--Newton approximation of the Hessian of $f$ at $y$.

We are going to describe explicitly the matrix of this map in a convenient basis of $T_y\V_r$.
For a non-zero complex vector $v\in \C^{n}$, we define the inner product: $\forall u,u' \in \C^{n}$,
$$
\apolar{u,u'}_{v} = \Re\big(u^{*} u' + (d-1) (u^{*} v) (v^{*} u') \|v\|^{-2} \big)
$$
It is a positive definite inner product on $\C^{n}\sim\R^{2n}$ since
$\apolar{u,u}_{v}= \|u\|^{2} + (d-1)  |(v^{*} u)|^{2}  \|v\|^{-2} \ge 0$ and it vanishes iff $u=0$. Notice that $\apolar{v,v}_{v}= d \|v\|^{2}$. The symmetric matrix associated to this inner product in the canonical basis of $\R^{2n}$ is
$$
M_{v} := I_{2n} + (d-1) \|v\|^{-2} ( v_{R}^{\ } v_{R}^{t} +v_{I}^{\ } v_{I}^{t})
$$
where $v_{R}=(\Re(v);\Im(v)), v_{I}=(-\Im(v);\Re(v))$ are the vectors of $\R^{2n}$ obtained by concatenating the real and imaginary part (resp. opposite imaginary and real part) of $v\in \C^{n}$.\\
Let $u_1=\frac{v_R}{\|v\|}$ and $u_2=\frac{v_I}{\|v\|}$. We notice that $u_{1}$ and $u_{2}$ can be completed to an orthonormal basis of $\mathbb{R}^{2n}$. Let $U$ denotes the matrix of this basis i.e. $U=[u_1, \dots, u_{2n}]$. Then $UU^t=I_{2n}$, so that an eigenvalue decomposition of the symmetric matrix $M_v$ of $\apolar{\cdot, \cdot}_{v}$ in the canonical basis of $\R^{2n}$ can be written as follows:
\begin{equation}\label{EVD}
M_v=U \mathrm{diag}(1+(d-1), 1+(d-1), 1, \dots, 1) U^t=U \mathrm{diag}(d, d, 1, \dots, 1) U^t.
\end{equation}
For shortness, we denote the strictly positive diagonal matrix $\mathrm{diag}(d, d, 1, \dots, 1)$ by $S$.

\begin{lemma}\label{lem:orthonormal} Let $v \neq 0\in \C^{n}\sim \R^{2n}$ and $p = (v^{t} \x)^{d}\in \C[\x]_{d}$.
Let $u_{1}, \ldots, u_{2n} \in \C^{n}$ be an orthonormal $\R$-basis for the inner product $\apolar{\cdot,\cdot}_{v}$ with $u_{1}={v\over \sqrt{d} \|v \|}$.
Then
$$
\bm q_{i}=\sqrt{d} \|v\|^{-d+1}  (u_{i}^{t}\x) (v^{t} \x)^{d-1}, i= 1, \ldots, 2n
$$
is an orthonormal basis of $T_{p} \Vnd$ for the inner product $\apolar{\cdot,\cdot}_{d}^{\R}$.
\end{lemma}
\begin{proof}
Using the apolar identities in lemma \ref{ap}, we have
\begin{eqnarray*}
{\apolar{\bm q_{i}, \bm q_{j}}_{d}^{\R}}
&=& \sqrt{d} \|v\|^{-d+1}\Re \big(\apolar{(u_{i}^{t}\x) (v^{t} \x)^{d-1}, \bm q_{j}}_{d}\big) \\
&=&  \sqrt{d^{-1}} \|v\|^{-d+1} \Re \big((u_{i}^{*} \nabla_{\x} \bm q_{j})(\overline{v}) \big) \\
&=&  \|v\|^{-2d+2}\Re \big((u_{i}^{*} u_{j}) (v^{*}v)^{d-1} +
(d-1) (u_{i}^{*}v) (v^{*}u_{j}) (v^{*}v)^{d-2}
\big)\\
&=& \Re\big((u_{i}^{*} u_{j}) +
(d-1) ({u}_{i}^{*} v) (v^{*}u_{j}) \|v\|^{-2}\big)= \apolar{u_{i}, u_{j}}_{v}.
\end{eqnarray*}
We deduce that $\apolar{\bm q_{i}, \bm q_{j}}_{d}^{\R}=\delta_{i,j}$ and $(\bm q_{i})_{i=1, \ldots, 2n}$ is an orthonormal basis of $T_{(v^{t}\x)^{d}} \Vnd$ for the inner product $\apolar{\cdot, \cdot}_{d}^{\R}$.
\end{proof}

We describe now how to compute an orthonormal basis for $\apolar{\cdot, \cdot}_{v}$.
\begin{lemma}
Let $M_{v} = U S U^{t}$ be the eigenvalue decomposition of $M_v$ as in (\ref{EVD}).
Let $\hat{u}_{1}= \sqrt{S^{-1}} U^{t} {v_{R}\over \sqrt{d} \|v \|}$ and let $Q\in \R^{2n \times 2n}$ be the orthogonal factor of a rank-revealing QR-decomposition of
$I_{2n} - \hat{u}_{1}^{\ } \hat{u}_{1}^{t} = Q R P$
where R is upper triangular and $P$ is a permutation matrix. Let
$$
u_{R,1}={v_{R}\over \sqrt{d} \|v \|}, u_{R,i}= U \sqrt{S^{-1}} Q_{[:,i-1]} \quad i=2, \ldots, 2n.
$$
Then the orthonormal $\R$-basis $u_{1}, \ldots,u_{2n}\in \C^{n}$ for $\apolar{\cdot, \cdot}_{v}$ is such that
$u_{i}= (u_{R,i})_{[1:n]} +$ $\ib \, (u_{R,i})_{[n+1:2n]}$\linebreak$\in \C^{n}$ for $i=1, \ldots, 2n$.
\end{lemma}
\begin{proof}
As $M_{v} = U S U^{t}$ with $U U^{t} = I_{2n}$ and $S \in \R^{2n\times 2n}$ a strictly positive diagonal matrix, we have $\sqrt{S^{-1}} U^{t} M_{v } U \sqrt{S^{-1}}=I_{2n}$. Thus the column vectors of $U \sqrt{S^{-1}}$ form an orthonormal basis of $\R^{2n}$ for $\apolar{\cdot, \cdot}_{v}$.

The vector $\hat{u}_{1}= \sqrt{S^{-1}} U^{t} {v_{R}\over \sqrt{d} \|v \|}$ is representing the vector ${v_{R}\over \sqrt{d} \|v \|}$ in this orthonormal basis.

The first $2n-1$ columns of the factor $Q$ in a rank-revealing QR-decomposition of
$I_{2n} - \hat{u}_{1}^{\ } \hat{u}_{1}^{t} = {Q} {R} P$
are orthonormal vectors $\hat{u}_{2},\ldots,\hat{u}_{2n}$ for $\apolar{\cdot, \cdot}_{v}$, expressed in the basis $U \sqrt{S^{-1}}$.
An orthonormal basis $u_{R,1}, u_{R,2}, \ldots, u_{R,2n}\in \R^{2n}$ for $\apolar{\cdot, \cdot}_{v}$ is thus given by
$u_{R,1}={v_{R}\over \sqrt{d} \|v \|}, u_{R,i}= U \sqrt{S^{-1}} Q_{[:,i-1]}, i=2, \ldots, 2n$.
The corresponding vectors $\in \C^{n}$ are
$u_{i}= (u_{R,i})_{[1:n]} +$ $\ib \, (u_{R,i})_{[n+1:2n]}$\linebreak$\in \C^{n}$ for $i=1, \ldots, 2n$.
\end{proof}

Notice that when $v$ is real and $u,u'$ are real such that $\apolar{v,u}=\apolar{v,u'}=0$,   $\apolar{u, u'}_{v} = \apolar{u,u'}$ is the standard inner product of $u,u'$. Consequently in the real case, an orthonormal basis $(u_{i})_{i=1, \ldots, n}\subset \R^{n}$ can be obtained directly from
$u_{1}={v\over  \|v \|}$ and a rank-revealing QR-decomposition of $I_{n}- u_{1}^{\ } u_{1}^{t}$.

For $y=(y_1,\dots,y_r)\in \V_{r}$ with $y_i=(v_i^t\x)^d\in\Vnd$, $\forall 1\le i\le r$, let
$(\bm q_{i,j})_{j=1, \ldots, 2n}$ be the orthonormal basis associated to $v_{i}$ defined in lemma \ref{lem:orthonormal}
 and
let $Q_{i}=[\bm q_{i,1}, \ldots,$ $\bm q_{i,2n}]\in \R^{2 s_{n,d}\times 2n} $ be the coefficient matrix of the polynomials $(\bm q_{i,j})_{j=1, \ldots, 2n}$ in the canonical $\R$-basis of $\C[\x]_{d}$.
The columns of the matrix
$$
Q=\mathrm{diag}(Q_i)_{1\le i\le r},
$$
represent an orthonormal basis of $T_y\V_r$ for the inner product induced by $\apolar{\cdot, \cdot }_{d}^{\R}$ on each component.

Therefore, the Jacobian matrix $J$ of $\sigma_r$ at $y$, which is the matrix associated to $D\sigma_r(y)=DF(y)$, with respect to the orthonormal basis $Q$ on $T_y\V_r$ and the standard real basis on $\C[\x]_d$ is given by:
$$
J=[Q_1,\dots,Q_r]\in\R^{2s_{n,d}\times 2nr}.
$$

\begin{proposition}
The Gauss--Newton equation (\ref{gnv}) in the orthonormal basis $Q$ of $T_y\V_r$
is of the form
$$
H \, \tilde{\eta} = - G,
$$
where $\tilde{\eta}^t=(\tilde{\eta}_{1}^t, \ldots, \tilde{\eta}_{r}^t) \in\R^{2nr}$ is the unknown coordinate vector of an element of the tangent space $T_{y}\V_{r}$ in the basis $Q$
and
\begin{itemize}
 \item $G=[G_{k}]_{k=1, \ldots, 2nr}$ with for $1\le i\le r,~1\le j\le 2n$,\\
$
G_{2 n (i-1)+j}=\sqrt{d^{-1}} \|v_{i}\|^{-d+1}\Big(d \sum_{k=1}^{r} \Re\big((u_{i,j}^{*} v_{k})(v_{i}^{*}v_{k})^{d-1})
 -\Re\big( u_{i,j}^{*}\nabla_{\x} \bm p(\overline{v_i}) \big)\Big)
$,
 \item $H=[H_{k,k'}]_{1\le k,k'\le2nr}$ with for $1\le i,i'\le r,~1\le j,j'\le 2n$,\\
$H_{2 n (i-1)+j, 2 n (i'-1)+j' }=\|v_{i}\|^{-d+1} \|v_{i'}\|^{-d+1}\Big( \Re\big((u_{i,j}^{*}u_{i',j'}) (v_{i}^{*}v_{i'})^{d-1}\big)
   + (d-1) \Re\Big(\big(u_{i,j}^{*}v_{i'}) (v_{i}^{*}u_{i',j'})$\linebreak$  (v_{i}^{*}v_{i'})^{d-2}\big)\Big)$.

\end{itemize}
\end{proposition}
\begin{proof}
As the matrix of $D\sigma_r(y)=DF(y)$ in the orthonormal basis $Q$ on $T_y\V_r$
and the standard real basis on $\C[\x]_d$ is $J$, we have
that the Gauss--Newton equation (\ref{gnv})
is $H \tilde{\eta} = - G$ with
\begin{itemize}
 \item $G=J^t vec(\sigma_r(y)-\bm p) = (\apolar{\bm q_{i,j}, \sigma_r(y)-\bm p}_{d}^{\R})$,
 \item $H=J^tJ = [Q_{1}, \ldots, Q_{r}]^{t}[Q_{1}, \ldots, Q_{r}] =
(\apolar{\bm q_{i,j}, \bm q_{i',j'}}_{d}^{\R})$.
\end{itemize}
By the apolar identities in lemma \ref{ap}, we have
\begin{eqnarray*}
\apolar{\bm q_{i,j}, \sigma_r(y)-\bm p}_{d}^{\R} & = &
\sqrt{d^{-1}} \|v_{i}\|^{-d+1} \Big(\sum_{k=1}^{r}\Re(u_{i,j}^{*} \nabla_{\x}(v_{k}^{t} \x)^{d}(\overline{v_i}) -u_{i,j}^{*}\nabla_{\x} \bm p(\overline{v_i})) \Big)\\
& = &
\sqrt{d^{-1}} \|v_{i}\|^{-d+1}\Big(\sum_{k=1}^{r}\Re\big( d (u_{i,j}^{*} v_{k})(v_{i}^{*}v_{k})^{d-1} - u_{i,j}^{*}\nabla_{\x} \bm p(\overline{v_i}) \big)\Big).
\end{eqnarray*}
Similarly,
\begin{eqnarray*}
\lefteqn{\apolar{\bm q_{i,j}, \bm q_{i',j'}}_{d}^{\R}}\\
 & = &
 \|v_{i}\|^{-d+1} \|v_{i'}\|^{-d+1} \Re\big(u_{i,j}^{*} \nabla_{\x}((u_{i',j'}^{t}\x) (v_{i'}^{t} \x)^{d-1})(\overline{v_i})\big)\\
 & = &
 \|v_{i}\|^{-d+1} \|v_{i'}\|^{-d+1} \Re\big((u_{i,j}^{*}u_{i',j'}) (v_{i}^{*}v_{i'})^{d-1}
+ (d-1) (u_{i,j}^{*}v_{i'}) (v_{i}^{*}u_{i',j'})  (v_{i}^{*}v_{i'})^{d-2}
\big),
\end{eqnarray*}
which ends the proof of the proposition.
\end{proof}
The Gauss--Newton equation
$$
H \, \tilde{\eta} = - G,
$$
solved in local coordinate with respect to the basis $Q$, yields a vector $\tilde{\eta}=(\tilde{\eta}_{1}; \ldots; \tilde{\eta}_{r})$ $\in\R^{2nr}$.
The components of the tangent vector $\eta=(\eta_1,\dots,\eta_r)\in T_y\V_r\in \C[\x]_{d}$ are then
$$
\eta_i=\sqrt{d}||v_i||^{-d+1}(v_i^t\x)^{d-1}\sum_{k=1}^{2n}{\tilde{\eta}_{i,k}(u^t_{i,k}\x)}, \quad i=1, \ldots, r.
$$

\subsubsection{Retraction on the Veronese manifold}\label{retraction1}

We define the retraction of a tangent vector $\eta \in T_{y}\V_{r}$ to a new point $\tilde{y}$ on the manifold $\V_{r}$ as follows:
$$
\tilde{y}=(\tilde{y}_1,\dots,\tilde{y}_r)=(R_{y_1}({\eta}_1),\dots,R_{y_r}({\eta}_r)),
$$
where $R_{y_i}:T_{y_i}\V^{n,d}\to\V^{n,d}$ is a retraction operator on the Veronese manifold for $i\in\left\{1,\dots,r\right\}$ that we describe hereafter (see lemma \ref{decomposeret}). The retraction on the Veronese manifold that we are going to describe is implemented directly on $y_i+\eta_i$ in the ambient space $\C[\x]_d$ without considering any tensor compression techniques on the symmetric tensor associated to $y_i+\eta_i$, as is elaborated for instance in \cite{Kressnerl_SV_2013} for multilinear tensors.

We will use the following matrix construction to define the retraction on $\V^{n,d}$.
\begin{definition}
The Hankel matrix of degree $(k, d-k)$ associated to a polynomial $\bm p$ in $\C[\x]_d$ is given by:
   \begin{center}
   $H_{\bm p}^{k,d-k}=(\apolar{\bm p,\x^{\alpha+\beta}}_d)_{|\alpha|=k, |\beta|=d-k}$.
       \end{center}
\end{definition}
This matrix is also known as the {\em  Catalecticant matrix} of the symmetric tensor $\bm p$
in degree $(k,d-k)$ or the {\em  flattening} of $\bm p$ in degree $(k,d-k)$.
In this definition, we implicitly assume that we have chosen a monomial ordering (for instance the lexicographic ordering on the monomials indexing the rows and columns of $H_{\bm p}^{k,d-k}$) to build the Hankel matrix. The properties of Hankel matrices that we will use are independent of this ordering.
Such a matrix is called a Hankel matrix since, as in the classical case, the entries of the matrix depend on the sum of the exponents of the monomials indexing the corresponding rows and columns. 

When $k=1$, using the apolar relations $\apolar{\bm p,x_{i} \x^{\beta}}_d= \frac{1}{d}\apolar{\partial_{x_{i}}\bm p, \x^{\beta}}_{d-1}$, we see that $H_{\bm p}^{1,d-1}$ is nothing else than the transposed of the coefficient matrix of the gradient $\frac{1}{d}\nabla_{\x} \bm p$ in the basis $\left({\x^{\beta}}{ \genfrac(){0pt}{1}{d-1}{\beta}^{-1}}\right)_{ | \beta | =d-1}$. When $\bm p=(v^{t} \x)^{d}\in \V^{n,d}$, $H_{\bm p}^{1,d-1}$ can thus be written as the rank-1 matrix $v \otimes (v^{t} \x)^{d-1}$.

Our construction of a retraction on $\Vnd$ is described in the following definition.
\begin{definition}\label{Pi}
For $v\in \C^{n}\setminus \{0\}$, let $\pi_{v}:\C[\x]_d \rightarrow \Vnd$ be the map such that $\forall \bm q\in \C[\x]_d$,
\begin{equation}\label{eq:proju}
  \pi_{v}(\bm q) = \frac{ \apolar{\psi(v), \bm q}_{d}}{\| \psi(v) \|_{d}^{2}} \psi(v),
\end{equation}
where $\psi: v\in \C^{n} \mapsto (v^{t} \x)^{d}\in \Vnd$ is the parametrization of the Veronese variety.
For $\bm p\in \C[\x]_d$, let $\theta(\bm p)\in \C^{n}$ be the first left singular vector of $H_{\bm p}^{1,d-1}$.
For $\bm p\in \V^{n,d}$, let
\begin{eqnarray*}
 R_{\bm p}: T_{\bm p}\Vnd & \rightarrow & \Vnd\\
\bm q & \mapsto &  \pi_{\theta(\bm p+\bm q)}(\bm p+\bm q).
\end{eqnarray*}
\end{definition}
The retraction that we are going to describe on the Veronese manifold is closely related to the one on the Segre manifold used in \cite{Breiding2018}. In fact, since the Segre manifold coincides with the manifold of tensors of multilinear rank $(1, \ldots, 1)$, the retraction in \cite{Breiding2018} is deduced from the truncated multilinear rank $(1, \ldots, 1)$ HOSVD of a real multilinear tensor, i.e. from the truncated rank one SVD of the matricization in the different modes \cite{Kressner_SV_2013}. For a symmetric tensor, the matricization with respect to any mode gives the same Catalecticant matrix in degree (1, d-1). Hereafter, we show, by different techniques, that a single truncated SVD of the Catalecticant matrix in degree $(1,d-1)$ gives a retraction on the Veronese manifold.

By the apolar identities, we check that $R_{\bm p}(\bm q)= (\bm p(\bar{u}) + \bm q(\bar{u}))\, (u^t\x)^d$ where $u= \theta(\bm p+\bm q)$. We also verify that $\pi_{\lambda\, u}=\pi_{u}$
for any $\lambda\in \C\setminus \{0\}$ and any $u\in \C^{n} \setminus \{0\}$.

By the relation \eqref{eq:proju}, for any $v\in \C^{n}\setminus \{0\}$, $\pi_{v}(\bm q)$
is the vector on the line spanned by $\psi(v)$, which is the closest to $\bm q$ for the apolar norm.
In particular, we have $\pi_{v}(\psi(v)) = \psi(v)$.

We verify now that $R_{\bm p}$ is a retraction on $\V^{n,d}$.

\begin{lemma}\label{fixe}
Let $\bm p\in\Vnd$. Then, $\bm p$ is a fixed point by $\pi_{u}$ where $u$ is the first left singular vector of $H_{\bm p}^{1, d-1}$.
\end{lemma}
\begin{proof}
If $\bm p=(v^{t} \x)^{d}=\psi(v) \in \V^{n,d}$ with $v\in \C^{n}\setminus \{0\}$, then the first left singular vector $u$ of $H_{\bm p}^{1, d-1}$ is up to a scalar equal to $v$. Thus we have $\pi_{u}(\bm p)=\pi_{v}(\psi(v)) = \psi(v) = \bm p$.
\end{proof}

\begin{proposition}\label{smooth}
Let $\bm p\in\Vnd$. There exists a neighborhood $\U_{\bm p}\subset\C[\x]_d$ of $\bm p$ such that the map $\rho: \bm q \in \U_{\bm p} \mapsto \pi_{\theta(\bm q)}(\bm q)$ is well-defined and $C^\infty$ smooth.
\end{proposition}
\begin{proof}
Let $\bm p\in\Vnd$ and  $\theta:\bm q\in \hp\rightarrow q\in \C^n$ where $q$ is the first left singular vector of the SVD decomposition of $H_{\bm q}^{1,d-1}$. Let $\gamma:\C[\x]_d\to\Vnd = \psi \circ \theta$ be the composition map by the parametrization map $\psi$ of $\V^{n,d}$.

By construction, we have $\rho: \bm q\mapsto \apolar{\bm q, \gamma(\bm q)}_d\, \gamma(\bm q)$.
Let $\O$ denotes the open set of homogeneous polynomials $\bm q\in \C[\x]_d$ such that the Hankel matrix $H_{\bm q}^{1,d-1}$ has a nonzero gap between the first and the second singular values. It follows from \cite{Chern2000SmoothnessAP} that the map $\theta$ is well-defined and smooth on $\O$. As $\bm p$ is in $\Vnd$ and $H_{\bm p}^{1,d-1}$ is of rank $1$, $\bm p\in \O$. Let $\U_{\bm p}$ be a neighborhood of $\bm p$ in $\C[\x]_d$  such that $\psi_{\vert\U_{\bm p}}$ is well-defined and smooth.
As the apolar product $\langle\cdot, \cdot\rangle_d$ and the multiplication are well-defined and smooth on $\C[\x]_d\times\C[\x]_d$, $\rho$ is well-defined and smooth on $\U_{\bm p}$, which ends the proof.
\end{proof}

As $\psi: v\in \C^{n} \mapsto (v^{t} \x)^{d}\in \Vnd$ is a parametrization of the Veronese variety $\Vnd$,
the tangent space of $\V^{n,d}$ at a point $\psi({v})$ is spanned by the first order vectors $D\psi(v)\, q$ of the Taylor expansion of
$\psi(v+ t\, q) = \psi(v)+ t\, D\psi({v})\, q + O(t^{2})$ for $q\in \C^{n}$.
We are going to use this observation to prove the rigidity property of $R_{p}$.
\begin{proposition}\label{prop:rigidity}
For $\bm p\in \V^{n,d}, \bm q\in T_{\bm p}(\V^{n,d})$,
$$
  \bm p+t\, \bm q -  R_{\bm p}(t\,\bm q) = O(t^{2}).
$$
\end{proposition}
\begin{proof}
As $\bm p\in \V^{n,d}, \bm q\in T_{\bm p}\V^{n,d}$,
there exist $v,q\in \C^{n}$ such that $\bm p=\psi(v), \bm q = D\psi(v)\, q$. In particular, we have
$\bm p+t\, \bm q -\psi(v+t\, q) = O(t^{2})$. This implies that $H_{\bm p+t\, \bm q}^{1, d-1} -H_{\psi(v+t\, q)}^{1,d-1}= O(t^{2})$.
By differentiability of simple non-zero singular values and their singular vectors \cite{Stewart2001MatrixAV}, we have $u_{t} - v_{t} = O(t^{2})$
where $u_{t}=\theta({\bm p+t\, \bm q})$ and $v_{t}=\theta({\psi(v+t\, q)})$ are respectively the first left singular vectors of
$H_{\bm p+t\, \bm q}^{1, d-1}$ and $H_{\psi(v+t\, q)}^{1,d-1}$.

Since $H_{\psi(v+t\, q)}^{1,d-1}$ is a matrix of rank $1$ and its image is spanned by $v+t\,q$, $v_{t}$ is a non-zero scalar multiple of $v+t\, q$ and we have $\pi_{v_{t}}= \pi_{v+t\, q}$.
By continuity of the projection on a line, we have
$$
\pi_{u_{t}}(\bm p+t\,\bm q) = \pi_{v_{t}}(\bm p+t\,\bm q) + O(t^{2})  =  \pi_{v+t\, q}(\bm p+t\,\bm q) + O(t^{2}).
$$
Since $\psi(v+t\,q)=\psi(v) + t\, D\psi(v) q +O(t^{2})= \bm p+t\,\bm q +O(t^{2})$, we have
$$
\pi_{v+t\, q}(\bm p+t\,\bm q)  = \pi_{v+t\,q}(\psi(v+t\, q)) + O(t^{2})
= \psi(v+t\, q) + O(t^{2}).
$$
We deduce that
\begin{eqnarray*}
\bm p+t\, \bm q -  R_{\bm p}(t\,\bm q) &=& \bm p+t\, \bm q -  \pi_{u_{t}}(\bm p+t\,\bm q) \\
& = & \bm p+t\,\bm q - \psi(v+t\, q) + \left(\psi(v+t\, q)-\pi_{v_{t}}(\bm p+t\,\bm q)\right)\\&& + \left(\pi_{v_{t}}(\bm p+t\,\bm q) - \pi_{u_{t}}(\bm p+t\,\bm q)\right)\\
& = & \psi(v) + t\, D\psi(v) q - \psi(v+t\, q) + O(t^{2}) = O(t^{2}),
\end{eqnarray*}
which proves the proposition.
\end{proof}

\begin{proposition}\label{retraction}
  Let $\bm p\in\Vnd$.
  The map $R_\bm p:T_\bm p\Vnd\rightarrow \Vnd$,~~$\bm q\mapsto R_{\bm p}(\bm q)=\pi_{\theta(\bm p+\bm q)}(\bm p+\bm q)$ is a retraction operator on the Veronese manifold $\Vnd$.
\end{proposition}
\begin{proof}
We have to prove that $R_{\bm p}$ verifies the three properties in \cref{conditions}.
%The map $R_{\bm p}:T_{\bm p}\Vnd\mapsto \Vnd$ can be defined as $R_{\bm p}=\pi\circ S_{\bm p}$, where $S_{\bm p}:T_{\bm p}\Vnd\mapsto\hp$,~$\bm q\mapsto \bm p+\bm q$. Let $\bm p\in\Vnd$.
\begin{enumerate}
    \item $R_{\bm p}(0_{\bm p})=\pi_{\theta(\bm p)}(\bm p+0_{\bm p})=\pi_{\theta(\bm p)}(\bm p)=\bm p$, by using lemma \ref{fixe}.
    \item Let $S_{\bm p}:T_{\bm p}\Vnd\rightarrow\C[\x]_d$,~$\bm q\mapsto \bm p+\bm q$.
      The map $S_{\bm p}$ is well-defined and smooth on $T_{\bm p}\Vnd$.
By proposition \ref{smooth}, $\pi$ is well-defined and smooth in a neighborhood $\U_{\bm p}$ of $\bm p\in\Vnd$. Thus $R_{\bm p}= \rho \circ S_{\bm p}$ is well-defined and smooth in a neighborhood $\U'_{\bm p} \subset T\Vnd$ of $0_{\bm p}$.
\item By proposition \ref{prop:rigidity},
%   Let $\rho_{\Vnd}: \U_{\bm p}\rightarrow\Vnd$, $\bm p\in \U_{\bm p}\mapsto \bm p_{best}$ be the projection map onto $\Vnd$. As $T_{\bm p}\Vnd$ is the first order approximation of $\Vnd$ around $\bm p\in T_{\bm p}\Vnd$, it follows that
%       $||(\bm p+t\bm q)-\rho_{\Vnd}(\bm p+t\bm q)||_d=O(t^2)~\text{for}~t\to0$.
% Using \cref{zebde} we have
$$
(\bm p+t\bm q)-R_{\bm p}(t\,\bm q) = O(t^2),
$$
which implies that $\frac{d}{dt}R_{\bm p}(t\,\bm q)\mid_{t=0}=\bm q$, or equivalently $DR_{\bm p}(0_{\bm p}) \bm q=\bm q$. Therefore we have $DR_{\bm p}(0_{\bm p})=id_{T_{\bm p}\Vnd}$.
\end{enumerate}
\end{proof}

\subsection{Adding a trust-region scheme (c.f. \cite[Chapter 7]{AbsMahSep2008})}\label{trust}
%(Riemannian trust-region \cite[Ch.7]{AbsMahSep2008}).
Recall that Riemannian Newton (resp. Gauss--Newton) is looking for a critical point of a real-valued function $f$, without distinguishing between local minimizer, saddle point and local maximizer. Furthermore, the convergence of this algorithm may not occur from the beginning. For these reasons, a trust region scheme is usually added to such algorithm in order to enhance the algorithm, with the desirable properties of convergence to a local minimum, with a local superlinear rate of convergence. In fact, trust region method ensures that $f$ decreases at each iteration, which reinforces, when convergence occurs, the possibility of finding a local minimizer. Nevertheless, a global convergence to a local minimizer from any initial points is not guaranteed even after adding the trust region scheme (see \cite[Subsection 7.4.1]{AbsMahSep2008} for the global convergence of Riemannian trust-region methods). We prove in proposition \ref{conv_RNETR} that under regularity assumptions, a local convergence for the Riemannian--Newton algorithm with trust region scheme can be obtained. Studying global or further local convergence properties for the Riemannian Newton (resp. Gauss Newton) method with trust region scheme for the STA problem is beyond the scope of this article.

Let $\M$ denote the Riemannian manifold $\N_r$ in subsection \ref{formulation} (resp. $\V_r$ in subsection \ref{ver}), and let $y_k\in\M$. The idea is to approximate the objective function $f$ to its second order Taylor series expansion in a ball of center $0_{y_k}\in T_{y_k}\M$ and radius $\Delta_k$ denoted by $B_{\Delta_k}:=\left\{\eta\in T_{y_k}\M\mid||\eta||\le\Delta_k\right\}$, and to solve the subproblem
\begin{equation}\label{sub}
   \min_{\eta\in B_{\Delta_k}}m_{y_k}(\eta),
\end{equation}
where $m_{y_k}(\eta):=f(y_k)+G_k^t\eta+\frac{1}{2}\eta^t H_k \eta$, $G_{k}$ is the gradient of $f$ at $y_{k}$ and $H_{k}$ is respectively the Hessian of $f$ at $y_{k}$ for the Riemannian Newton method and the Gauss--Newton approximation of $f$ at $y_k$ for the Riemannian Gauss--Newton method.

By solving \eqref{sub}, we obtain a solution $\eta_k\in T_{y_k}\M$. Accepting or rejecting the candidate new point $y_{k+1}=R_{y_k}(\eta_k)$ is based on the quotient $\rho_k=\frac{f(y_k)-f(y_{k+1})}{m_{y_k}(0)-m_{y_k}(\eta_k)}$.\\
If $\rho_k$ exceeds 0.2 then the current point $y_k$ is updated, otherwise the current point $y_k$ remains unchanged.\\
The radius of the trust region $\Delta_k$ is also updated based on $\rho_k$. We choose to update the trust region as in \cite{Breiding2018} with a few changes.

Let $\Delta_{y_0}:=10^{-1}\sqrt{\frac{d}{r}\sum_{i=1}^{r}{||w_i^0||^2}}$ in the Riemannian Newton iteration (resp. $\Delta_{y_0}:=10^{-1}$\linebreak$\sqrt{\frac{d}{r}\sum_{i=1}^{r}{||v_i^0||^{2d}}}$ in the Riemannian Gauss--Newton iteration), $\Delta_{\mathrm{max}}:=\frac{1}{2}||\bm p||_d$.
We take the initial radius as $\Delta_0=\mathrm{min}\{\Delta_{y_0},\Delta_{\mathrm{max}}\}$, if $\rho_k>0.6$ then the trust region is enlarged as follows: $\Delta_{k+1}=\mathrm{min}\{2||\eta_k||,\Delta_{\mathrm{max}}\}$. Otherwise the trust region is shrinked by taking $\Delta_{k+1}=\mathrm{min}\{(\frac{1}{3}+\frac{2}{3}(1+e^{-14(\rho_k-\frac{1}{3})})^{-1})\Delta_k, \Delta_{\mathrm{max}}\}$.

We choose the so-called dogleg method to solve the subproblem \eqref{sub} \cite{GVK502988711}.
Let $\eta_N$ be the Newton direction given by $H\eta_N=-G$, let $\eta_c$ denote the Cauchy point given by $\eta_c=-\frac{G^tG}{G^tHG}G$,
and let $\eta_I$ be the intersection of the boundary of the sphere $B_\Delta$ and the vector pointing from $\eta_c$ to $\eta_N$.
Then the optimal solution $\eta^*$ of \eqref{sub} by the dogleg method is given as follows:
\begin{equation*}
\eta^*=\begin{cases} \eta_N & \text{if}~ ||\eta_N||\le\Delta\text{,}\\
-\frac{\Delta}{||G||}G & \text{if}~ ||\eta_N||>\Delta ~\text{and}~ ||\eta_c||\ge\Delta\text{,}\\
\eta_I & \text{otherwise.}\end{cases}\end{equation*}

The algorithm of the Riemannian Newton (resp. Gauss--Newton) method with trust region scheme for the STA problem is denoted by RNE-N-TR (resp. RGN-V-TR) and is given in pseudo-code by \cref{alg:rim2}.

\begin{algorithm}[ht]
\caption{Riemannian Newton (resp. Gauss--Newton) algorithm with trust region sheme for the STA problem ``RNE-N-TR''(resp. ``RGN-V-TR'')}
\label{alg:rim2}
\begin{algorithmic}
\STATE{\textbf{Input:} The homogeneous polynomial $\bm p\in\C[\x]_d$ associated to the symmetric tensor to approximate, $r<r_g$}.\\
\STATE{\textbf{Choose} initial point $y_{0}\in\N_r$ (resp. $y_{0}\in\V_r$).}
\WHILE{the method has not converged}
\STATE{\scriptsize{
1.}~\normalsize Compute the gradient vector and the Hessian matrix (resp. Gauss--Newton Hessian approximation);}
\STATE{\scriptsize{2.}~\normalsize Solve the subproblem \eqref{sub} for the search direction $\eta_k\in B_{\Delta_k}$ by using the dogleg method;}
\STATE{\scriptsize{3.}~\normalsize Compute the candidate next new point $y_{k+1}=R_{y_k}(\eta_k)$; }
\STATE{\scriptsize{4.}~\normalsize Compute the quotient $\rho_k$; }
\STATE{\scriptsize{5.}~\normalsize Accept or reject $y_{k+1}$ based on the quotient $\rho_k$;}
\STATE{\scriptsize{6.}~\normalsize Update the trust region radius $\Delta_{k}$.}
\ENDWHILE
\STATE{\textbf{Output:} $y_{*} \in \N_{r}$ (resp. $y_{*}\in \V_{r}$).}

\end{algorithmic}
\end{algorithm}
The \cref{alg:rim2} is stopped when $\Delta_k\le\Delta_{\mathrm{min}}$ (by default $\Delta_{\mathrm{min}}=10^{-3}$), or when the maximum number of iterations exceeds $N_{\mathrm{max}}$.
\begin{remark}\rm
  In order to handle ill-conditioned Hessian (resp. Gauss--Newton Hessian approximation) matrices in \cref{alg:rim2}, we use the Moore-Penrose pseudoinverse  \cite{doi:10.1137/1.9781611971484,1585,doi:10.1137/0905030}. This can appear in cases where some vectors $v_{i}$ of the rank-$r$ approximation span close lines, which yields a singularity problem in the iteration. In particular, this is the case when the symmetric border rank of the symmetric tensor is not equal to its symmetric rank \cite{Breiding2018,Comon2008}, \cite[section 2.4]{landsberg2011tensors}. For example, the tensor $\bm p=(v_0^t\x)(v_1^t\x)^{d-1} + \epsilon\, T$, with $v_0, v_1\in\R^n$, $T\in \R[\x]_d$ and $\epsilon$ very small, is close to the tensor $(v_0^t\x)(v_1^t\x)^{d-1}= \lim_{\delta \rightarrow 0}\frac{1}{d\, \delta}(((v_1+\delta v_0)^{t} \x)^{d} - (v_{1}^{t}\x)^{d})$ of border rank 2 and symmetric rank $d$. It can be very well approximated by a tensor of rank $2$, with two vectors of almost the same direction.
\end{remark}

Under some regularity assumption, it is possible to guarantee  that RNE-N-TR algorithm converges to a local minimum of the distance function $f$.
\begin{proposition}\label{conv_RNETR} Let $\bm p\in\hp$, let $\bm p_{0} \in \Sigma_{r}$ be the initial point of RNE-N-TR and let $B_{0}=B(\bm p,||\bm p-\bm p_{0}||_{d})$ be the ball of center $\bm p$ and radius $||\bm p-\bm p_{0}||_{d}$ in $\hp$. Assume that $B_{0}\cap \sec_{r} \subset \Sigma_{r}^{\mathrm{reg}}$ (i.e. all points of $\sec_{r}$ in $B_{0}$ are non-defective), then RNE-N-TR converges to a local minimum $y\in \N_r$ of the distance function $f$ to $\Sigma_{r}$.
\end{proposition}
\begin{proof}
  Let $\Sigma_{r}^{0}:= B_{0}\cap \sec_{r} = B_{0}\cap \Sigma_{r}^{\mathrm{reg}}$ be the set of non-defective tensors of rank $r$ in $B_{0}$.
  As $\varphi_{r}: (W,V_{R},V_{I})\in \N_{r} \mapsto \sum_{i=1}^{r} w_{i} ((v_{R,i}+\ib\, v_{I,i})^{t}\x)^{d}\subset \sec_{r}\subset \C[\x]_{d}$ is locally injective at a non-defective tensor, it defines a local diffeomorphism between $\N_{r, 0}=\varphi_{r}^{-1}(\Sigma_{r}^{0})$ and  $\Sigma_{r}^{0}$. As $B_{0}$ is compact, $\N_{r, 0}=\varphi_{r}^{-1}(\Sigma_{r}^{0})$ is a compact Riemannian manifold. By construction, the distance between $\bm p$ and the iterates $\bm p_{i}$ is decreasing  in RNE-N-TR, so that their decomposition is in $\N_{r, 0}=\varphi_{r}^{-1}(\Sigma_{r}^{\mathrm{reg}}\cap B_{0})$.
As $\N_{r, 0}$ is a compact Riemannian manifold and $f$ is smooth on $\N_{r,0}$ (as a polynomial function), \cite[Corollary 7.4.6]{AbsMahSep2008} implies that the iterates of RNE-N-TR of Riemannian Newton method with a trust region sheme on $\N_{r, 0}$ converge to a local minimum of the distance function $f$.
\end{proof}
The regularity assumption $B_{0}\cap \sec_{r} \subset \Sigma_{r}^{\mathrm{reg}}$
implies that the ball centered at $\bm p$ and containing the initial point of the iteration does not contain a defective tensor. In this case, the iterates, which distance to $\bm p$ decreases, remain in the ball and the limit decomposition is a non-defective low rank tensor. This assumption, satisfied
when $\bm p$ is far enough from the singular locus of $\sec_{r}$, is a sufficient condition to ensure the regularity of the iteration points and their limit.

\section{Numerical experiments}\label{num}
In this section, we present four numerical experiments using the RNE-N-TR and RGN-V-TR algorithms. These algorithms are implemented in the package \texttt{TensorDec.jl}\footnote{It can be obtained from \url{https://gitlab.inria.fr/AlgebraicGeometricModeling/TensorDec.jl} and run in Julia version 1.1.1.
  See functions rne\_n\_tr and rgn\_v\_tr.}. We use a Julia implementation for the method SPM tested in subsection \ref{first}.
The solvers from Tensorlab v3 \cite{tensorlab3.0} are run in MATLAB 7.10.
The experimentation was done on a Dell Windows desktop with 8 GB memory and Intel Core i5-5300U, 2.3 GHz CPU.

\subsection{Choice of the initial point}\label{initial}
The choice of the initial point is a crucial step in iterative methods.
% We consider two cases.
% In the first case, the initial point $(W_0,V_0)$ is chosen such that $W_0=(w_i^0)_{1\le i\le r}\in\R^r$ and $V_0=[v_i^0]_{1\le i\le r}\in\mathbb{K}^{n\times r}$, where $\mathbb{K}=\R~\text{or}~ \C$,  randomly according to an uniform or a normal distribution.
%
We use the direct algorithm of \cite{harmouch:hal-01440063}, based on the computation of generalized eigenvectors and generalized eigenvalues of pencils of Hankel matrices (see also \cite{mourrain:hal-01367730}), to compute an initial rank-$r$ approximation. This algorithm, denoted SMD, works only with $r<r_g$ such that $\iota\le\lfloor\frac{d-1}{2}\rfloor$ where $\iota$ denotes the interpolation degree of the points in the rank-$r$ decomposition \cite[Chapter 4]{eisenbud_geometry_2005}. This implies that $r<\binom{n+d'-1}{d'}$ where $d'=\lfloor\frac{d-1}{2}\rfloor$.
It first computes a SVD decomposition of the Hankel matrix of the tensor $\bm t$ in degree $\big(\lfloor\frac{d-1}{2}\rfloor, d-\lfloor\frac{d-1}{2}\rfloor\big)$, extracts the first $r$ singular vectors, computes a simultaneous diagonalisation of the matrices of multiplication by the variables $x_{i}$ by taking a random combination of them, computing its eigenvectors and deducing the points and weights in the approximate decomposition of $\bm t$.
The rationale behind choosing the initial point with this method is when the symmetric tensor is already of symmetric rank $r$ with $r<r_g$ and $\iota\le\lfloor\frac{d-1}{2}\rfloor$, then this computation gives a good numerical approximation of the exact decomposition, so that the Riemannian Newton (resp. Gauss--Newton) algorithm needs few iterations to converge numerically.
We will see in the following numerical experiments that this initial point is an efficient choice to get a good low rank approximation of a symmetric tensor.

\subsection{Best rank-1 approximation and spectral norm}\label{second}
Let $\bm p\in\S^d(\R^n)$, a best real rank-1 approximation of $\bm p$ is a minimizer of the optimization problem
\begin{equation}\label{1.6}
   \textup{dist}_1(\bm p) :=   \min_{\bm t\in\S^d(\R^{n}), \rank_s(\bm t)=1} ||\bm p-\bm t||_d^2   = \min_{(w,v)\in\R\times\mathbb{S}^{n-1}} ||\bm p-w(v^t\x)^d||_d,
\end{equation}
where $\SS^{n-1}=\{v\in \R^n\mid ||v||=1\}$ is the unit sphere.
This problem is equivalent to $\min_{\bm t\in\T^d(\R^{n}), \rank(\bm t)=1} ||{\bm p}-\bm t||_F^2$  since at least one global minimizer is a symmetric rank-1 tensor \cite{doi:10.1137/110835335}.

The real spectral norm of $\bm p\in\S^{d}(\R^n)$, denoted by $||\bm p||_{\sigma,\R}$ is by definition:
\begin{equation}\label{1.8}
||\bm p||_{\sigma,\R}^2 := \max_{v\in\mathbb{S}^{n-1}}|\bm{p}(v)|.
\end{equation}
  The two problems \eqref{1.6} and \eqref{1.8} are related by the following equality:
  \begin{equation*}
    \textup{dist}_1(\bm p)^2 =||\bm p||_d^2-||\bm p||_{\sigma,\R}^2,
  \end{equation*}
  which we deduce by simple calculus and properties of the apolar norm (see also \cite{doi:10.1137/S0895479898346995,doi:10.1137/110835335}):
\begin{eqnarray*}
  \textup{dist}_1(\bm p)^2 &=&\min_{(w,v)\in\R\times\mathbb{S}^{n-1}}  ||\bm p-w(v^t\x)^d||_d^2\\
  &=& \min_{(w,v)\in\R\times\mathbb{S}^{n-1}}||\bm p||_d^2-2\langle \bm p,w(v^t\x)^d\rangle_d+||w(v^t\x)^d||_d^2\\
  &=&\min_{(w,v)\in\R\times\mathbb{S}^{n-1}}||\bm p||_d^2-2w\,\bm p(v)+w^2\\
  &=&\min_{v\in\mathbb{S}^{n-1}}||\bm p||_d^2-\vert \bm p(v)\vert ^2=||\bm p||_d^2-\max_{v\in\mathbb{S}^{n-1}}|\bm p(v)|^2=||\bm p||_d^2-||\bm p||_{\sigma,\R}^2.
  \end{eqnarray*}
  Therefore, if $v$ is a global maximizer of \eqref{1.8} such that $w=\bm p(v)$, then $w\, v^{\otimes d}$ is a best rank-1 approximation of $\bm p$. Herein, a rank-1 approximation $w\,  v^{\otimes d}$, such that $w=\bm p(v)$ and $||v||=1$, is better when $|w|$ is higher. Therefore, in the following experimentation, we report the weight $w$ obtained by the different methods.

  In \cite{doi:10.1137/130935112} the authors present an algorithm called ``SDP" based on semidefinite relaxations to find a best real rank-1 approximation of a real symmetric tensor by finding a global optimum of $\bm p$ on $\SS^{n-1}$. We choose two examples from \cite{doi:10.1137/130935112}, on which we apply the RNE-N-TR with initial point chosen according to the SMD algorithm adapted for $1\times 1$ matrices.
% Nevertheless, we apply several random combinations and we choose between them the point $(w, v)$ of biggest weight in absolute value $|w|$ such that $||v||=1$ to be the initial point.
%   In order to test if this strategy can find best rank-1 approximation in these two examples.
The reason behind using RNE-N-TR instead of RGN-V-TR is to take advantage of the local quadratic rate of convergence that distinguishes the exact Riemannian Newton iteration in RNE-N-TR \cite[Theorem 6.3.2]{AbsMahSep2008}. We compare these methods with the method CCPD-NLS which is a non-linear least-square solver for the symmetric decomposition from Tensorlab v3 \cite{tensorlab3.0} in MATLAB 7.10, where we run 50 instances (i.e. 50 random initial points obeying Gaussian distributions), and we take the absolute value of the weight in average for this method.

We denote by $|w_{\mathrm{sdp}}|$ (resp. $|w_{\mathrm{rne}}|$) the weight in absolute value given by SDP (resp. RNE-N-TR), and $|w_{\mathrm{ccpd}}|$ denotes the absolute value of the weight in average given by CCPD-NLS. Note that $|w_{\mathrm{sdp}}|$ is the spectral norm of $\bm p$, since SDP gives a best rank-1 approximation. We report the time spent by SDP from \cite{doi:10.1137/130935112} (resp. RNE-N-TR including the computation time of the initial point) in seconds (s) and we denote it by $t_{\mathrm{sdp}}$ (resp. $t_{\mathrm{rne}}$). We denote by $N_{\mathrm{rne}}$ the number of iterations in RNE-N-TR. We denote by $d_0$ the norm between $\bm p$ and the initial point of RNE-N-TR, and by $d_*$ the norm between $\bm p$ and the solution obtained by RNE-N-TR. We denote by $t_{\mathrm{ccpd}}$ (resp. $N_{\mathrm{ccpd}}$) the time in seconds (s) (resp. number of iterations) in average for CCPD-NLS.
  %Recall that the RNE algorithm is an exact Riemannian Newton method, and then it has a local quadratic rate of convergence \cite[Theorem 6.3.2]{AbsMahSep2008}

\begin{example}\label{example6}\rm\cite[Example 3.5]{doi:10.1137/130935112}. Consider the tensor $\bm p\in\S^3(\R^n)$ with entries:
  $$(\bm p)_{i_1, i_2, i_3}=\frac{(-1)^{i_1}}{i_1}+\frac{(-1)^{i_2}}{i_2}+\frac{(-1)^{i_3}}{i_3},$$
  corresponding to the polynomial $\bm p= \sum_{|\alpha|=3}(\sum_{i=1}^{n} \alpha_{i} \frac{(-1)^{i}}{i})\,\genfrac(){0pt}{1}{3}{\alpha}\, \x^{\alpha}$.\\
\end{example}
\begin{example}\rm\label{example8}\cite[Example 3.7]{doi:10.1137/130935112}. Consider the tensor $\bm p\in\S^5(\R^n)$ given as:
  $$ (\bm p)_{i_1,...,i_5}=(-1)^{i_1}\log(i_1)+ \cdots +(-1)^{i_5}\log(i_5),
  $$
corresponding to the polynomial $\bm p= \sum_{|\alpha|=5}(\sum_{i=1}^{n}\alpha_{i} (-1)^{i}\log(i))\,\genfrac(){0pt}{1}{5}{\alpha}\, \x^{\alpha}$.\\
\end{example}

\begin{table}[ht!]
  \scriptsize
  \centering
  \caption{Comparison of RNE-N-TR, CCPD-NLS and SDP for Example \ref{example6} and Example \ref{example8}. \label{rank1}}
  \resizebox{\textwidth}{!}{
    \begin{tabular}{|c?c|c|c|c|c?c|c|c|c|c|} \hline
     &\multicolumn{5}{c?}{Example \ref{example6}}& \multicolumn{5}{c|}{Example \ref{example8}}\\\hline n& 10&20&30&40& 50&5&10&15&20& 25\\\thickhline
     $|w_{\mathrm{rne}}|$ &\textbf{17.8} &\textbf{34.2} &\textbf{50.1} &\textbf{65.9} &\textbf{81.6}&\textbf{1.100e+2} &\textbf{8.833e+2} &\textbf{2.697e+3} &\textbf{6.237e+3} &\textbf{11.504e+3} \\\hline $d_0$&32.4 &28.4 &44 &64.6 &78.3&526.1 &6.559e+3 &26.318e+3 &64.268e+3 &132.213e+3 \\\hline $d_*$&13.2 &28.3 &43.8 &59.5 &75.3&477.5 &6.096e+3 &24.643e+3 &60.435e+3 &121.892e+3 \\\hline
      $t_{\mathrm{rne}}$&0.038 &0.304 &1.5 &3.3 &12.1&0.058 &0.282 &3.8 &18.3 &34.8
      \\\hline $N_{\mathrm{rne}}$&5 &4 &4 &4 &6&5 &4 &6 &6 &6 \\\thickhline  $|w_{\mathrm{ccpd}}|$&14.0&29.3&43.3&60.0&75.6&78.9&8.68e+2&2.354e+3&6.148e+3&10.587e+3\\\hline$t_{\mathrm{ccpd}}$&0.173&0.109&0.105&0.122&0.143& 0.093&0.187&1.2&5.5&16.7\\\hline $N_{\mathrm{ccpd}}$&27&25&22&23&22&19&29&16&23&17\\\thickhline
     $|w_{\mathrm{sdp}}|$&\textbf{17.8} &\textbf{34.2} &\textbf{50.1} &\textbf{65.9} &\textbf{81.6}&\textbf{1.100e+2} &\textbf{8.833e+2} &\textbf{2.697e+3} &\textbf{6.237e+3} &\\\hline
     $t_{\mathrm{sdp}}$&2.0&6.0&30.0&245.0&1965.0&1.0&22.0&78.0&1350.0&\\\hline

  \end{tabular}}
\end{table}

The results in \Cref{rank1} show that the RNE-N-TR finds a global minimizer, starting
from the initial point given by the SMD algorithm. The RNE-N-TR algorithm converges to this point in few iterations, and with very reduced time compared to the SDP algorithm especially when $n$ grows. On the other hand, $|w_{\mathrm{ccpd}}|$ is smaller than  $|w_{\mathrm{sdp}}|$, implying that CCPD computes, in several cases, a local minimum, which is not a global minimum i.e. a best rank-1 approximation. In comparison for these cases, RGN-V-TR took more iterations ($\sim$20) than RNE-N-TR and consequently more time, while reaching the same optimimum.

% Thus, as RNE-N-TR is an iterative method expected to give local solutions as CCPD-NLS.
The fact that RNE-N-TR finds the best rank-1 approximation in these examples comes from the good initial point provided by SMD algorithm. However, we have no guarantee that RNE-N-TR with this initial point will always converge to a best rank-1 approximation.
% for example, for $n=25$ in \cref{example8} we compute $|w_{RNE}|$, but since we don't have $|w_{sdp}|$, we cannot verify if $|w_{RNE}|$ corresponds to a best rank-1 approximation or not.
This experimentation shows that RNE-N-TR combined with SMD algorithm for the initial point is an efficient method to get a good real rank-1 approximation of a real symmetric tensor.

\subsection{Symmetric rank-\texorpdfstring{$r$}{Lg} approximation}\label{third} We consider two examples of a real and a complex valued sparse symmetric tensors, in order to compare the performance of RNE-N-TR and RGN-V-TR
% with initial point computed by SMD,
with state-of-the-art non-linear least-square solvers CCPD-NLS and SDF-NLS for symmetric decomposition from Tensolab v3 with random initial point following a standard normal distribution. These solvers employ factor matrices as parameterization and use a Gauss--Newton method with dogleg trust region steps called ``NLS-GNDL''. We fix 200 iterations as maximal number of iterations, and we run 50 instances for these methods and we report the minimal, median and maximal residual error denoted `$\mathrm{err}$', such that, $\mathrm{err}:=||\bm p-\bm p_*||_d$, where $\bm p$ is the symmetric tensor to approximate and $\bm p_*$ is the approximate symmetric tensor of rank-$r$. In the computation of the initial point by SMD algorithm in RNE-N-TR and RGN-V-TR, we compute eigenvectors of a random linear combination of multiplication operators. This computation is sensitive to the choice of the linear combination, when the operators are not commuting, which explains why we report also the minimal, median and maximal $\mathrm{err}$ for these two methods. The average of time t is in seconds, and the average number of iterations N is rounded to the closest integer.

\begin{example}\label{example10}Let $\bm p\in\S^3(\R^{10})$ such that:
  \begin{equation*}
(\bm p)_{i_1, i_2, i_3}=\begin{cases} i_1^2+1 & \text{if}~ i_1=i_2=i_3,\\
1 & \text{if}~ [i_1,i_2, i_3]\equiv [i,i,j] \text{ with } i \neq j,
\\
0 & \text{otherwise.}\end{cases}\end{equation*}
($[i_1,i_2, i_3]\equiv [j_{1},j_{2},j_{3}]$ iff there exists a permutation $\sigma\in S_{3}$ such that $[i_{\sigma(1)}, i_{\sigma(2)}, i_{\sigma(3)}] $ $= [j_{1}, j_{2}, j_{3}]$).
This sparse symmetric tensor corresponds to the polynomial $\bm p=\sum_{i=1}^{10} i^{2} x_{i}^{3} + (\sum_{i=1}^{10} x_{i}^{2})\times (\sum_{i=1}^{10} x_{i})$.
\end{example}
\begin{example}\label{example11}Let $\bm p\in\S^3(\C^{10})$ such that:\begin{equation*}
(\bm p)_{i_1, i_2, i_3}=\begin{cases} e^{\sqrt{i_1}+i_1^2 \sqrt{-1}}+\frac{i_1}{10}\sqrt{-1} & \text{if}~ i_1=i_2=i_3,\\
 \frac{i}{10} \sqrt{-1}& \text{if}~ [i_1,i_2, i_3]\equiv [i,i,j] \text{ with } i \neq j,\\
0 & \text{otherwise.}\end{cases}\end{equation*}
This sparse symmetric tensor corresponds to the polynomial $\bm p=\sum_{i=1}^{10} e^{\sqrt{i}+i^2 \sqrt{-1}}\, x_{i}^{3} $ $+ $\linebreak$\sqrt{-1}(\sum_{i=1}^{10} \frac{i}{10} x_{i}^{2}) \times (\sum_{i=1}^{10} x_{i})$.
\end{example}

\begin{table}[ht]
    \scriptsize
\centering
\caption{Comparison of
  \textbf{RNE-N-TR}, \textbf{RGN-V-TR},
  \textbf{CCPD-NLS},
  \textbf{SDF-NLS} for Examples \ref{example10} and \ref{example11}.\label{table4.3}
}
\begin{tabular}{|c|c|c|c|c|c|}
  \hline
  \multicolumn{6}{|c|}{Example \ref{example10}}\\
  \thickhline
  \multirow{2}{*}{r} &\multicolumn{3}{c|}{$\mathrm{err}_{\mathrm{rne}}$} & %
  \multicolumn{1}{c|}{$t_{\mathrm{rne}}$}&\multicolumn{1}{c|}{$N_{\mathrm{rne}}$}\\
\cline{2-6}
   &{min}&{med}&{max}&{avg}&{avg}
  \\\hline
  3&70.6 &96 &134.3 &0.03 &2
  \\\hline
  5&33.3 &54.2 &91.8 &0.08&3
    \\\hline
  10&0.884 &0.884 &94.1 &0.465 &6\\\thickhline
  \multirow{2}{*}{r} & \multicolumn{3}{c|}{$\mathrm{err}_{\mathrm{rgn}}$} &
  \multicolumn{1}{c|}{$t_{\mathrm{rgn}}$}&\multicolumn{1}{c|}{$N_{\mathrm{rgn}}$}\\\cline{2-6}
  &{min}&{med}&{max}&{avg}&{avg}\\
  \hline
    3&70.6 &96 &136.8 &0.064 &3
    \\\hline
    5&33.3& 48.8& 105.3& 0.149& 4
    \\\hline
    10&0.886& 0.886& 10.1&0.836& 7\\\thickhline

\multirow{2}{*}{r} & \multicolumn{3}{c|}{$\mathrm{err}_{\mathrm{ccpd}}$} & %
\multicolumn{1}{c|}{$t_{\mathrm{ccpd}}$}&\multicolumn{1}{c|}{$N_{\mathrm{ccpd}}$}\\\cline{2-6}
&{min}&{med}&{max}&{avg}&{avg}\\
\hline
  3&71 &102 &137.1 &0.067 &14
  \\\hline
  5&34.2 &54.7 &121 &0.116 &26
  \\\hline
  10&7.8 &7.8 &9.7 &0.5 &90
  \\\thickhline
 \multirow{2}{*}{r}   &\multicolumn{3}{c|}{$\mathrm{err}_{\mathrm{sdf}}$} & \multicolumn{1}{c|}{$t_{\mathrm{sdf}}$}&\multicolumn{1}{c|}{$N_{\mathrm{sdf}}$}\\
\cline{2-6}
&{min}&{med}&{max}&{avg}&{avg}
  \\\hline
  3&71 &96.3 &136 &0.155 &14
  \\\hline
  5&34.2 &49.4 &105.3 &0.212&16
  \\\hline
  10&7.8 &8.2 &38.3 &2.3 &158
  \\\hline
\end{tabular}
\begin{tabular}{|c|c|c|c|c|c|}
  \hline
  \multicolumn{6}{|c|}{Example \ref{example11}}\\
\thickhline
  \multirow{2}{*}{r} &\multicolumn{3}{c|}{$\mathrm{err}_{\mathrm{rne}}$} & %
\multicolumn{1}{c|}{$t_{\mathrm{rne}}$}&\multicolumn{1}{c|}{$N_{\mathrm{rne}}$}\\
\cline{2-6}
   &{min}&{med}&{max}&{avg}&{avg}
  \\\hline
  3&22.4 &28.8 &30.9 &0.04& 2
  \\\hline
  5&14.1& 17.4& 24.6& 0.07& 3
  \\\hline
  10&0.164& 0.168& 0.369& 0.113& 2
  \\\thickhline
  \multirow{2}{*}{r} & \multicolumn{3}{c|}{$\mathrm{err}_{\mathrm{rgn}}$} &
  \multicolumn{1}{c|}{$t_{\mathrm{rgn}}$}&\multicolumn{1}{c|}{$N_{\mathrm{rgn}}$}\\\cline{2-6}
  &{min}&{med}&{max}&{avg}&{avg}\\
  \hline
    3&22.4& 27.6& 36.1& 0.065& 3
    \\\hline
    5&14.1& 17.1& 24.6& 0.101& 3
    \\\hline
    10&0.162& 0.164& 0.169& 0.219& 2\\\thickhline

\multirow{2}{*}{r} & \multicolumn{3}{c|}{$\mathrm{err}_{\mathrm{ccpd}}$} & %
\multicolumn{1}{c|}{$t_{\mathrm{ccpd}}$}&\multicolumn{1}{c|}{$N_{\mathrm{ccpd}}$}\\
\cline{2-6}
  &{min}&{med}&{max}&{avg}&{avg}\\
\hline
  3&22.9 &26.8 &35.2 &0.084 &14
  \\\hline
  5&14.9 &17 &26.6 &0.104 &18
  \\\hline
  10&4.8 &4.8 &11.2 &0.506 &60
  \\\thickhline
  \multirow{2}{*}{r}  &\multicolumn{3}{c|}{$\mathrm{err}_{\mathrm{sdf}}$} & \multicolumn{1}{c|}{$t_{\mathrm{sdf}}$}&\multicolumn{1}{c|}{$N_{\mathrm{sdf}}$}\\
\cline{2-6}
&{min}&{med}&{max}&{avg}&{avg}
  \\\hline
3&22.9 &27.4 &35.2 &0.254 &15
  \\\hline
  5&14.9 &17.8 &26.5 &0.35&19
  \\\hline
  10&4.8 &6.2 &12.6 &2.5 &144
  \\\hline
\end{tabular}
\end{table}

The numerical results in \Cref{table4.3} show that the number of iterations of RNE-N-TR and RGN-V-TR method is low compared to the other methods. The iterations in RNE-N-TR and RGN-V-TR are more expensive. The numerical quality of approximation is better for RNE-N-TR and RGN-V-TR than the other methods in this test. It is of the same order as the other methods for $r=3,5$ but much better for $r=10$. This can be explained by the fact that the initial point provided by SMD method is close to a good rank-10 approximation. Notice that when $r=3, 5$ the initial point provided by SMD method, based on truncated SVD and eigenvector computations, yields the same behavior as a random initial point (a random linear combination of the matrices of a pencil is used to compute the eigenvectors in SMD method).

\subsection{Approximation of perturbations of low rank symmetric tensors}\label{first}
In this section, we consider perturbations of random low rank tensors. For a given rank $r$, we choose $r$ random vectors $v_{i}$ of size $n$, obeying Gaussian distributions and compute the symmetric tensor $\bm t=\sum_{i=1}^{r} (v_{i}^{t} \x)^{d}$  of order $d$. We choose a random symmetric tensor $\bm t_{\mathrm{err}}$ of order $d$, with coefficients also obeying Gaussian distributions, normalize it so that its apolar norm is $\epsilon$ and add it to $\bm t$: $\tilde{\bm t}= \bm t + \epsilon \, \frac{\bm t_{\mathrm{err}}}{\|\bm t_{\mathrm{err}}\|_{d}}$. We apply the different approximation algorithms to $\tilde{\bm t}$ and compute the relative error factor ref $:=\frac{\|\bm t_{*}-\bm t\|_{d}}{\epsilon}$ between the approximation $\bm t_{*}$ of rank $r$ computed by the algorithm and the rank-$r$ tensor $\bm t$. We run this computation for $100$ random instances and report the geometric average of the relative error. The average number of iterations $N$ is rounded to the closest integer, and the average time $t$ is in seconds.

As the initial tensor $\tilde{\bm t}$ is in a ball of radius $\epsilon$ centered at the tensor $\bm t$ of rank $r$, we expect $\bm t_{*}$ to be at distance to $\bm t$ smaller than $\epsilon$ and the relative error factor to be less than $1$.

We compare the RNE-N-TR and RGN-V-TR methods with the initial point computed by SMD algorithm, with the recent Subspace Power Method (SPM) of \cite{KileelSubspacepowermethod2019} and the state-of-the-art implementation CPD-NLS of the package Tensorlab v3. Note that CPD-NLS is designed for the canonical polyadic decomposition \cite{doi:10.1002/sapm192761164}. Nevertheless, in practice it is often observed that applying a general tensor rank approximation method (like CPD-NLS) from a symmetric starting point will usually result in a symmetric approximation. Since CPD-NLS is an efficient tensor decomposition routine of Tensorlab v3, we choose to compare our methods with this algorithm in this numerical experiment, using symmetric initial points and verifying that the obtained tensor approximations are symmetric. As SPM works for even order tensors with real coefficients, the comparison in \Cref{table1} is run for tensors in $\S^4(\R^{10})$. In Table \ref{table2}, we compare CPD-NLS, RNE-N-TR, and RGN-V-TR for tensors in $\S^d(\C^{10})$ of order $d=4$ and with complex coefficients.
These tables also provide a numerical comparison with the low rank approximation methods tested in Example 5.4 of \cite{NieLowRankSymmetric2017}, since the setting is the same. We also run this tensor perturbation test on some complex examples in which the  approximation rank is higher than the mode size of the tensor (see \Cref{table3}). We test this with the three methods RNE-N-TR, RGN-V-TR, and CPD-NLS. We run 20 instances, for each example of tensor and $\epsilon$.

The computational time for the methods RNE-N-TR and RGN-V-TR includes the computation of the initial point by the SMD algorithm.
We fix 200 iterations as maximal number of iterations for RNE-N-TR, RGN-V-TR and CPD-NLS. For SPM, the iterations are stopped when the distance between two consecutive iterates is less than $10^{-10}$
or when the maximal number of iterations ($N_{}=400$ in this experimentation) is reached.

\begin{table}[ht!]
    \scriptsize
\centering
\caption{Computational results of \textbf{SPM}, \textbf{RNE-N-TR}, and \textbf{RGN-V-TR} for rank-$r$ approximations in $\S^4(\R^{10})$.\label{table1}}
\begin{tabular}{|c|c?c|c|c?c|c|c?c|c|c|}
\hline
{r} & {$\epsilon$}&{ref$_{\mathrm{spm}}$} & %
    {$t_{\mathrm{spm}}$}&{$N_{\mathrm{spm}}$}&{ref$_{\mathrm{rne}}$} & %
        {$t_{\mathrm{rne}}$}&{$N_{\mathrm{rne}}$}&{ref$_{\mathrm{rgn}}$} & %
            {$t_{\mathrm{rgn}}$}&{$N_{\mathrm{rgn}}$}\\\thickhline

\multirow{5}{*}{1}&$1$&0.103& 0.04& 28& 0.105& 0.07& 2& 0.11& 0.083& 3
 \\\cline{3-11}&$10^{-1}$&0.103& 0.039& 28& 0.104& 0.04& 2& 0.11& 0.069& 3
 \\\cline{3-11}
    &$10^{-2}$&0.1& 0.04& 28& 0.1& 0.04& 2& 0.103& 0.058& 2
 \\\cline{3-11}&$10^{-4}$&0.101& 0.041& 29& 0.101& 0.041& 2& 0.166& 0.044& 2
 \\\cline{3-11}&$10^{-6}$&0.104& 0.041& 30& 0.104& 0.041& 2& 0.17& 0.045& 2
 \\\thickhline\multirow{5}{*}{2}&$1$&0.15& 0.1& 69& 0.175& 0.137& 3& 0.159& 0.16& 3
\\\cline{3-11}&$10^{-1}$&0.15& 0.091& 65& 0.153& 0.076& 2& 0.159& 0.13& 3
 \\\cline{3-11} &$10^{-2}$&0.144& 0.086& 66& 0.149& 0.072& 2& 0.15& 0.111& 2
 \\\cline{3-11}&$10^{-4}$&0.148& 0.089& 66& 0.157& 0.076& 2& 0.199& 0.076& 2
 \\\cline{3-11}&$10^{-6}$&0.146& 0.087& 67& 0.151& 0.073& 2& 0.195& 0.073& 2
 \\ \thickhline\multirow{5}{*}{3}&$1$&0.185& 0.126& 109& 0.194& 0.172& 3& 0.194& 0.208& 3
\\\cline{3-11}&$10^{-1}$&0.185& 0.135& 111& 0.195& 0.128& 2& 0.195& 0.198& 3
 \\\cline{3-11} &$10^{-2}$&0.187& 0.119& 113& 0.208& 0.099& 2& 0.195& 0.175& 2
 \\\cline{3-11}&$10^{-4}$&0.182& 0.102& 106& 0.197& 0.092& 2& 0.217& 0.095& 2
 \\\cline{3-11}&$10^{-6}$&0.183& 0.101& 105& 0.196& 0.094& 2& 0.206& 0.097& 2
 \\\thickhline \multirow{5}{*}{4}&$1$&0.217& 0.159& 168& 0.25& 0.546& 8& 0.225& 0.278& 3
 \\\cline{3-11}&$10^{-1}$&0.218& 0.161& 168& 0.245& 0.319& 4& 0.228& 0.241& 3
 \\\cline{3-11}&$10^{-2}$&0.211& 0.163& 162& 0.241& 0.134& 2& 0.219& 0.239& 3
 \\\cline{3-11}&$10^{-4}$&0.216& 0.167& 169& 0.26& 0.128& 2& 0.261& 0.136& 2
 \\\cline{3-11}&$10^{-6}$&0.227& 0.167& 168& 0.259& 0.126& 2& 0.259& 0.133& 2
 \\\thickhline\multirow{5}{*}{5}&$1$&0.244& 0.207& 217& 0.339& 1.199& 13& 0.252& 0.594& 5
\\\cline{3-11}&$10^{-1}$&0.244& 0.221& 220& 0.255& 0.252& 2& 0.252& 0.317& 3
 \\\cline{3-11}&$10^{-2}$&0.247& 0.223& 218& 0.292& 0.175& 2& 0.254& 0.321& 3
  \\\cline{3-11}&$10^{-4}$&0.246& 0.215& 213& 0.304& 0.16& 2& 0.304& 0.165& 2
  \\\cline{3-11}&$10^{-6}$&0.249& 0.231& 226& 0.307& 0.158& 2& 0.311& 0.165& 2
  \\\hline
\end{tabular}
\end{table}
\begin{table}[ht!]
    \scriptsize
\centering
\caption{Computational results of \textbf{CPD-NLS}, \textbf{RNE-N-TR}, and \textbf{RGN-V-TR} for rank-$r$ approximations in $\S^4(\C^{10})$.\label{table2}}
\begin{tabular}{|c|c?c|c|c?c|c|c?c|c|c|}
\hline
{r} & {$\epsilon$}&{ref$_{\mathrm{cpd}}$} & %
    {$t_{\mathrm{cpd}}$}&{$N_{\mathrm{cpd}}$}&{ref$_{\mathrm{rne}}$} & %
        {$t_{\mathrm{rne}}$}&{$N_{\mathrm{rne}}$}&{ref$_{\mathrm{rgn}}$} & %
            {$t_{\mathrm{rgn}}$}&{$N_{\mathrm{rgn}}$}\\\thickhline

\multirow{5}{*}{1}&$1$&0.117 &0.05 & 10& 0.11& 0.06& 2& 0.115& 0.069& 3 \\\cline{3-11}&$10^{-1}$&0.118 &0.046 & 10& 0.108& 0.054& 2& 0.112& 0.084& 3\\\cline{3-11}
    &$10^{-2}$&0.116 &0.043 &10 &0.107& 0.044& 2& 0.11& 0.06& 2\\\cline{3-11}&$10^{-4}$&0.114 &0.042 &10& 0.107& 0.037& 2& 0.227& 0.038& 2 \\\cline{3-11}&$10^{-6}$&0.113 &0.037 &11 & 0.112& 0.036& 2& 0.237& 0.037& 2 \\\thickhline\multirow{5}{*}{2}&$1$&0.167  &0.072 &14&0.162& 0.078& 2& 0.166& 0.118& 3\\\cline{3-11}&$10^{-1}$&0.169 &0.077 &14 & 0.164& 0.063& 2& 0.167& 0.111& 3 \\\cline{3-11} &$10^{-2}$&0.162 &0.071 &14 & 0.163& 0.061& 2& 0.163& 0.09& 2\\\cline{3-11}&$10^{-4}$&0.171  &0.071 &14 & 0.163& 0.062& 2& 0.204& 0.063& 2 \\\cline{3-11}&$10^{-6}$&0.175 &0.069 &13 & 0.162& 0.062& 2& 0.23& 0.064& 2 \\ \thickhline\multirow{5}{*}{3}&$1$&0.201 &0.115 &16 & 0.204& 0.135& 2& 0.204& 0.163& 3\\\cline{3-11}&$10^{-1}$&0.223 &0.109 &17 & 0.206& 0.091& 2& 0.203& 0.157& 3\\\cline{3-11} &$10^{-2}$&0.228 &0.117 &17 & 0.209& 0.086& 2& 0.203& 0.152& 2\\\cline{3-11}&$10^{-4}$&0.202 &0.103 &15 & 0.205& 0.091& 2& 0.243& 0.093& 2 \\\cline{3-11}&$10^{-6}$&0.284 &0.124 &19 &0.211& 0.088& 2& 0.234& 0.091& 2\\\thickhline \multirow{5}{*}{4}&$1$&0.235 &0.149 &18 & 0.234& 0.192& 3& 0.234& 0.23& 3 \\\cline{3-11}&$10^{-1}$&0.232 &0.165 &19 & 0.244& 0.132& 2& 0.238& 0.215& 3 \\\cline{3-11}&$10^{-2}$&0.237 &0.142 &17 & 0.25& 0.113& 2& 0.232& 0.219& 3 \\\cline{3-11}&$10^{-4}$&0.238 &0.158 &19 & 0.25& 0.112& 2& 0.255& 0.117& 2 \\\cline{3-11}&$10^{-6}$&0.232 &0.161 &19 & 0.254& 0.111& 2& 0.274& 0.116& 2 \\\thickhline\multirow{5}{*}{5}&$1$&0.275 &0.21 &22 & 0.261& 0.269& 3& 0.261& 0.345& 3\\\cline{3-11}&$10^{-1}$&0.264 &0.186 &19 & 0.269& 0.211& 2& 0.261& 0.288& 3  \\\cline{3-11}&$10^{-2}$&0.266 &0.211 &22 & 0.305& 0.148& 2& 0.264& 0.292& 3 \\\cline{3-11}&$10^{-4}$&0.265 &0.169 &18 & 0.293& 0.158& 2& 0.299& 0.163& 2 \\\cline{3-11}&$10^{-6}$&0.266 &0.206 &21 & 0.298& 0.158& 2& 0.301& 0.161& 2\\\hline
\end{tabular}
\end{table}

In Tables \ref{table1}, \ref{table2}, the number of iterations of the RNE-N-TR and RGN-V-TR methods is significantly smaller than the number of iterations of the other methods. In SPM, the number of iterations to get an approximation of a single rank-$1$ term of the approximation is about $30$, indicating a practical linear convergence as predicted by the theory \cite[Theorem 5.10]{KileelSubspacepowermethod2019}.
As the method CPD-NLS is based on a quasi-Newton iteration, its local convergence is sub-quadratic, which also explains the relatively high number of iterations. The low number of iterations in RNE-N-TR and RGN-V-TR can be explained by the choice of the initial point by SMD algorithm. This provides a good initialization such that a solution by RNE-N-TR and RGN-V-TR can be obtained in a few number of iterations.

The cost of an iteration appears to be higher in RNE-N-TR and RGN-V-TR than in the other methods. Nevertheless, the total time is of the same order. Note that the cost of an iteration seems higher in RGN-V-TR than RNE-N-TR. Despite the fact that the first algorithm computes the Gauss--Newton approximation of the Hessian matrix, whereas the second algorithm computes the exact Hessian matrix. This can be explained by the use of a parametrization in the first algorithm (i.e. the Cartesian product of Veronese manifolds), which involves a more expensive retraction using SVD decomposition on larger matrices.

These experimentation also show a good numerical behavior for the Riemannian methods. In particular, the numerical quality of the low rank approximation is good for RNE-N-TR and RGN-V-TR, in comparison with SPM and CPD-NLS. The average of the relative error factor in RNE-N-TR and RGN-V-TR is less than $1$. The numerical results in \cite[Example 5.4]{NieLowRankSymmetric2017} for GP method and small perturbations ($\epsilon\in \{10^{-2},10^{-4},10^{-6}\}$), show that the numerical quality in GP-OPT method is worse than with these methods.
\begin{table}[ht!]
\scriptsize
\centering
\caption{Computational results of \textbf{CPD-NLS}, \textbf{RNE-N-TR}, and \textbf{RGN-V-TR}.\label{table3}}
\resizebox{\textwidth}{!}{
\begin{tabular}{|c|c?c|c|c|c?c|c|c|c?c|c|c|c|}
\hline
\multirow{3}{*}{{$\begin{array}{c}
d\\n\\r
\end{array}$\vspace{2mm}}}&\multirow{3}{*}{$\epsilon$}&\multicolumn{2}{c|}{ref$_{\mathrm{cpd}}$} & %
    {$t_{\mathrm{cpd}}$}&{ $N_{\mathrm{cpd}}$}&\multicolumn{2}{c|}{ref$_{\mathrm{rne}}$} & %
        {$t_{\mathrm{rne}}$}&{$N_{\mathrm{rne}}$}&\multicolumn{2}{c|}{ref$_{\mathrm{rgn}}$} & %
            {$t_{\mathrm{rgn}}$}&{ $\begin{array}{c}\\ N_{\mathrm{rgn}}\\ \ \end{array}$}\\
\cline{3-14}&
  &{min}&{max}&avg&avg &{min}&{max}&avg&avg &{min}&{max}&avg&avg\\
\thickhline
\multirow{4}{*}{$\begin{array}{c}
5\\4\\10
\end{array}
$}&$1$&0.803&7.8&2.1&162& 0.834& 25.7& 3.4& 273& 0.815& 1.1& 1.2& 49\\\cline{3-14}&$10^{-2}$&0.849&881.2&2.3&172&0.718& 0.933& 0.197& 16& 0.718& 0.933& 0.078& 4\\\cline{3-14}
    &$10^{-4}$&1.5&9.8e+4&2&157&0.711& 0.933& 0.0379& 3& 0.711& 0.933& 0.0535& 3\\\cline{3-14}
        &$10^{-6}$&776.4&1.9e+7&2&184&0.789& 0.912& 0.042& 3& 0.789& 0.912& 0.071& 4 \\\thickhline\multirow{4}{*}{$\begin{array}{c}
5\\15\\20
\end{array}
$}&$1$&0.15&1.9e+3&9.8&45&0.153& 0.172& 22.6& 3& 0.153& 0.172& 27.1& 3\\\cline{3-14}&$10^{-2}$&0.149&1.2e+5&12.7&62&0.151& 0.183& 13.6& 2& 0.148& 0.169& 26.9& 3 \\\cline{3-14} &$10^{-4}$&0.152& 9.2e+6&13.8&67&0.152& 0.181& 13.6& 2& 0.152& 0.181& 14.7& 2\\\cline{3-14} &$10^{-6}$&0.155 & 1.4e+9 &11.9 &59 &0.156& 0.173& 13.8& 2& 0.156& 0.174& 14.8& 2\\ \thickhline\multirow{4}{*}{$\begin{array}{c}
6\\5\\12\end{array}$}&$1$&0.515&109.7&1.4&61&0.467& 0.706& 0.342& 4& 0.467& 0.622& 0.353& 4\\\cline{3-14}&$10^{-2}$&0.519 & 2.4e+4 &3.8 &143 &0.472& 0.62& 0.155& 3& 0.472& 0.62& 0.205& 3\\\cline{3-14} &$10^{-4}$& 0.518&9.6e+5&2.9&137&0.493& 0.622& 0.222& 4& 0.493& 0.622& 0.35& 5\\\cline{3-14} &$10^{-6}$&1.1&9.7e+7&2.3&112&0.647& 0.6& 0.098& 2& 0.492& 0.591& 0.211& 3\\\thickhline \multirow{4}{*}{$\begin{array}{c}
7\\8\\15\end{array}$}&$1$&0.183
        &2.1e+3 &54.9 &46 &0.171& 0.21& 8.3& 3& 0.171& 0.21& 8.5& 3\\\cline{3-14}&$10^{-2}$&0.174& 5.3e+4&52.3&47&0.137& 0.171& 4.7& 2& 0.169& 0.201& 8.1& 3\\\cline{3-14}&$10^{-4}$&0.168&3.5e+6&63.1&54&0.138& 0.169& 4.5& 2& 0.138& 0.169& 4.7& 2\\\cline{3-14}&$10^{-6}$&0.179 &1.1e+9 &75.6 &65 &0.142& 0.177& 4.4& 2& 0.142& 0.177& 4.6& 2\\\hline
\end{tabular}}
\end{table}

We also compare CPD-NLS, RNE-N-TR and RGN-V-TR for perturbation of random tensors of rank $r>n$ and report the minimal and maximal relative error with the average number of iterations $N$ (rounded to the closest integer) and the average time $t$ (in seconds) in \Cref{table3}. The considered cases in \Cref{table3} are for the degree $d$, the number of variables $n$ and the rank $r$ such that $(d,n,r)$ is respectively (5,4,10), (5,15,20), (6,5,12), and (7,8,15).
We see that the maximal relative error factor {\em ref} reached by RNE-N-TR and RGN-V-TR with initial point by SMD is less than 1. There is an exception in the first case when $\epsilon=1$, where a large number of iterations is needed for RNE-N-TR and RGN-V-TR. On the other hand, the minimal relative error of CPD-NLS is less than 1 in almost all \Cref{table3}, whereas its maximal relative error is higher than 1 in all \Cref{table3}.

This numerical experiment indicates that for these examples of random low rank tensors with random noise, SMD provides a good initial point, close enough to a good solution, so that RNE-N-TR and RGN-V-TR need a few number of iterations. In this context, the combination of an adaptive choice of initial point and a Newton-type method is successful.

\subsection{Symmetric tensor with large differences in the scale of the weight vector}\label{tst4} Consider the case of a real symmetric tensor $\bm t=\sum_{i=1}^{r}{w_i(v_i^t\bm x)^d}$, $\|v_i\|=1$, $w_i>0$, with large differences in the scale of the weights $w_{i}$ i.e. $\frac{\mathrm{max}_iw_i}{\mathrm{min}_iw_i}$ is large. More precisely, there are large differences in the norms of the rank-1 symmetric tensors $w_i(v_i^t \bm x)^d$.
We randomly sample real symmetric tensors of order $d=3$ and dimension $n=7$ with $r\in\{5,10\}$, according to the following model:
\begin{equation*}
  \bm t=\sum_{i=1}^{r}{10^{\frac{is}{r}}(v_i^t\bm x)^d},~\|v_i\|=1.
\end{equation*}
The components of the weight vector increase exponentially from $10^{\frac{s}{r}}$ to $10^s$.

We aim to compare the performance of RNE-N-TR and RGN-V-TR methods (hereafter called respectively RNE and RGN for shortness) in this configuration. We run the following test:
\begin{itemize}
  \item Take $\bm t$ as above, and create a perturbated tensor $\bm t_p=\frac{\bm t}{\|\bm t\|}+10^{-5}\frac{\bm t_{\mathrm{err}}}{\|\bm t_{\mathrm{err}}\|}$, where $\bm t_{\mathrm{err}}\in\R[\x]_d$ is a random symmetric tensor with coefficients obeying Gaussian distributions;
  \item run 20 random initial points obeying Gaussian distributions;
  \item run RNE and RGN with a maximum of iterations $N_{\mathrm{max}}=500$, and report  in average respectively: the relative error (in geometric average) $\mathrm{err}_{\mathrm{rel}}:=\big\|\frac{\bm t}{\|\bm t\|}-\bm t_*\big\|_d$, where $\bm t_*$ is a rank-$r$ symmetric decomposition obtained by these methods; the number of iterations $N_{\mathrm{iter}}$; and the computation time $t$ in seconds (s). We also report the number $N_{\mathrm{opt}}$ of instances where $\mathrm{err}_{\mathrm{rel}}\le 1.1.10^{-5}$.
\end{itemize}
\begin{table}[ht]
\scriptsize
\centering
\caption{Computational results for RNE-N-TR and RGN-V-TR for scaled weights.\label{scale}}

\begin{tabular}{|c?c|c?c|c?c|c|} \hline
  \multicolumn{7}{|c|}{$r=5$}\\
  \thickhline
  $s$& \multicolumn{2}{c?}{$1$} &\multicolumn{2}{c?}{$2$}&\multicolumn{2}{c|}{$3$} \\
\thickhline
Alg & RNE & RGN & RNE & RGN & RNE & RGN\\\hline
$\mathrm{err}_{\mathrm{rel}}$ &0.456&5.8e-6&0.411&5.5e-6&0.246&1.4e-5\\\hline
 $N_{\mathrm{iter}}$&120&39&165&61&175&77 \\\hline
$t$ &2.0&1.1&2.4&1.4&2.5&1.8 \\\hline
$N_{\mathrm{opt}}$ &0&20&0&20&0&17\\\hline
\end{tabular}

\begin{tabular}{|c?c|c?c|c?c|c|} \hline
  \multicolumn{7}{|c|}{$r=10$}\\
  \thickhline
  $s$& \multicolumn{2}{c?}{$1$} &\multicolumn{2}{c?}{$2$}&\multicolumn{2}{c|}{$3$} \\
\thickhline
Alg & RNE & RGN & RNE & RGN & RNE & RGN\\\hline
$\mathrm{err}_{\mathrm{rel}}$ &0.372&9.4e-6&0.195&1.6e-6&0.224&6.8e-5 \\\hline
 $N_{\mathrm{iter}}$&423&87&270&186&392&206 \\\hline
$t$ &16.9&6.3&10.8&13.7&15.5&15.0 \\\hline
$N_{\mathrm{opt}}$ &0&18&0&20&0&9\\\hline
\end{tabular}
\end{table}

The results in \Cref{scale} show that RGN outperforms RNE. In fact, the average of the relative error in RGN is better, up to five order of magnitude, than in RNE. Moreover, starting from the same 20 random initial points in the two methods; RGN succeeded to reach an optimum, at least in 9 instances with the different order of scale $s$, while RNE could not find any optimum. Notice that, as we mentionned before, the cost of one iteration in RGN is higher than in RNE. The good performance of RGN compared to RNE in this test was expected, since the orthonormal basis of the tangent space computed in RGN method is independent of the weight factor. This behavior was also observed in \cite[Subsection 3.4]{Breiding2018} for real multilinear tensors, parametrized by
Segre manifolds.
%the authors addressed to this case with multilinear tensors. Where they proved that the Gauss-Newton method for the tensor rank decomposition, used with parametrization of the constraint set via rank-1 tensor manifolds called Segre manifolds (Veronese manifolds in the symmetric case), is not affected by large differences in norm between rank-1 tensors.
\section{Conclusion}
\label{sec:conclusions}
We presented two Riemannian Newton optimization methods for approximating a given complex-valued symmetric tensor by a low rank symmetric tensor. We used in subsection \ref{formulation} the weighted normalized factor matrices parametriz\-ation for the constraint set. We developed an exact Riemannian Newton iteration with exact computation of the Hessian matrix (RNE-N-TR). We exploited in subsection \ref{computation} the properties of the apolar product and of partial complex derivatives, to deduce a simplified and explicit computation of the gradient and Hessian of the square distance function in terms of the points, weights of the decomposition and the tensor to approximate. We proved that under some regularity conditions on non-defective tensors in the neighborhood of the initial point, the iteration is converging to a local minimum. In subsection \ref{ver}, we parametrized the constraint set via Cartesian product of Veronese manifolds. Taking into account the geometry of the Veronese manifold, we constructed a suitable basis for its tangent space at a given point on this manifold. Using this basis, we developed a Gauss--Newton iteration (RGN-V-TR). In subsection \ref{retraction1}, we presented a retraction operator on the Veronese manifold.
We showed that, combined with SMD method for choosing the initial point,
the two methods have a good practical behavior in several experiments: in subsection \ref{second} to compute a best real rank-1 approximation of a real symmetric tensor, in subsection \ref{third} to compute a low rank approximation of sparse symmetric tensors, and in subsection \ref{first} to compute low rank approximations of random perturbations of low rank symmetric tensors. In subsection \ref{tst4}, we showed that the numerical behavior of RNE-N-TR is affected by large differences in the scaling of the rank-1 symmetric tensor, where RGN-V-TR outperformed this algorithm in this case.

In future work, we plan to investigate the computation of initial points for the Riemannian Newton iterations applied to tensors of higher rank and the low rank approximation problem for other families of tensors, such as multi-symmetric or skew symmetric tensors.

%\appendix

%\section{Sample Appendix Section}

 
 %\bibliography{cas-refs}





% \end{thebibliography}
\section{Acknowledgement}
We would like to thank the anonymous reviewers for their valuable comments that improved this article.
\bibliographystyle{elsarticle-num} 
\bibliography{Tensor}

%\newpage
\appendix

\section{Computation details}\label{append}
This appendix gives first the proof of proposition \ref{prop:newton:nr} which relates the Riemannian gradient and Hessian to the real gradient and Hessian (\ref{proof1}), then the proof of the explicit formulas of the real gradient (\ref{proof2}) and Hessian (\ref{proof3}) stated respectively in propositions \ref{gradient} and \ref{hessian}.
\subsection{Proof of proposition \ref{prop:newton:nr}}\label{proof1}
Let $y=(w, v_1, \dots, v_r, v\sp{\prime}_1, \dots, v\sp{\prime}_r )\in\N_r$. Let $\mathcal{P}_y$ be the orthogonal projector on $T_y\N_r$. Let $Q\in \R^{(r+2nr)\times (r+(2n-1)r)}$ such that its columns form an orthonormal basis of the image of $\mathcal{P}_y$ or equivalently of $T_y\N_r$. As the Riemannian gradient of $f$ is the projection of $Df_R$, the first order differentials of $f_R$, on the tangent space $T_y\N_r$ \cite[Chapter 5]{AbsMahSep2008}, we have $G=Q^tG^R$, where $G^R$ is the vector which represents the classical first order partial derivatives of $f_R$ at $y^R$ in the canonical basis. \\ Let $\eta\in T_y\N_r$, $z\in T_y\N_r^\perp$. We have from \cite{RH} that the Riemannian Hessian matrix of $f$ at $y$ is given by the formula: $H\eta=\mathcal{P}_yH^R\eta+\mathfrak{U}_y(\eta,\mathcal{P}_y^\perp G^R)$, where  $H^R$ is the matrix of the second order derivatives of $f_R$ at $y^R$ in the canonical basis, $\mathfrak{U}_y$ is the Weingarten map on $\N_r$ at $y$ given by $\mathfrak{U}_y(\eta,z)=\mathcal{P}_yD_{\eta}\mathcal{P}z$, where $\mathcal{P}$ is a matrix valued function on $\N_r$ determined as follows: $\mathcal{P}:y\in\N_r\mapsto \mathcal{P}_y$, and $D_{\eta}\mathcal{P}z$ represent the time derivative of $y\mapsto\mathcal{P}_yz$ in terms of the time derivative of $y$ i.e. $\dot{y}\in T_y\N_r$ applied at $\dot{y}=\eta$, and $\mathcal{P}_y^\perp=I-\mathcal{P}_y$ is the orthogonal projector on $T_y\N_r^\perp$.\\
    As $y\in\N_r$ we have $w\in{\R_+^*}^r$, and $\dbl{v}_i:=(v_i, v\sp{\prime}_i)\in\mathbb{S}^{2n-1}$, $\forall 1\le i\le r$.
Let $u=(u_0, u_1, \dots$\linebreak$, u_r, u\sp{\prime}_1, \dots, u\sp{\prime}_r)\in\R^{r+2nr}$, such that $\dbl{u}_i=(u_i,u\sp{\prime}_i)$, $\forall 1\le i\le r$. Let $\mathcal{P}_{w}$ (resp. $\mathcal{P}_{\dbl{v}_i}$) denote the orthogonal projector on $T_w(\R^*_+)^r=\R^r$ (resp. $T_{\dbl{v}_i}\mathbb{S}^{2n-1}$), we have that:
    $\mathcal{P}_w(u_0)=u_0$, $\mathcal{P}_{\dbl{v}_i}\dbl{u}_i=(I_{2n}-\dbl{v}_i\dbl{v}_i^t)\dbl{u}_i$, $\forall 1\le i\le r$, thus:\\
$$
\mathcal{P}_yu=\begin{pmatrix}u_0 \\ ((I_{2n}-\dbl{v}_1\dbl{v}_1^t)\dbl{u}_1)[1:n] \\ \vdots \\ ((I_{2n}-\dbl{v}_r\dbl{v}_r^t)\dbl{u}_r)[1:n] \\  ((I_{2n}-\dbl{v}_1\dbl{v}_1^t)\dbl{u}_1)[n+1:2n] \\  \vdots \\ ((I_{2n}-\dbl{v}_r\dbl{v}_r^t)\dbl{u}_r)[n+1:2n]  \end{pmatrix}=\begin{pmatrix}u_0 \\ u_1-v_1\dbl{v}_1^t\dbl{u}_1\\ \vdots \\ u_r-v_r\dbl{v}_r^t\dbl{u}_r\\ u\sp{\prime}_1-v\sp{\prime}_1\dbl{v}_1^t\dbl{u}_1 \\ \vdots \\  u\sp{\prime}_r-v\sp{\prime}_r\dbl{v}_r^t\dbl{u}_r \end{pmatrix},
\mathcal{P}_y^\perp u=\begin{pmatrix}0_r \\ v_1\dbl{v}_1^t\dbl{u}_1 \\ \vdots \\ v_r\dbl{v}_r^t\dbl{u}_r\\v\sp{\prime}_1\dbl{v}_1^t\dbl{u}_1 \\ \vdots \\  v\sp{\prime}_r\dbl{v}_r^t\dbl{u}_r  \end{pmatrix}.
$$
Let $\mathfrak{U}_{\dbl{v}_i}$ be the Weingarten map on $\mathbb{S}^{2n-1}$ at $\dbl{v}_i$. For $\eta=(\eta_0, \eta_1, \dots, \eta_r, \eta\sp{\prime}_1, \dots, \eta\sp{\prime}_r)\in T_y\N_r$, and $z=(z_0, z_1, \dots, z_r, z\sp{\prime}_1, \dots, z\sp{\prime}_r)\in T_y\N_r^\perp$ with $\dbl{\eta}_i=(\eta_i,\eta\sp{\prime}_i)\in T_{\dbl{v}_i}\mathbb{S}^{2n-1}$ and $\dbl{z}_i=(z_i,z\sp{\prime}_i)\in {T_{\dbl{v}_i}{\mathbb{S}^{2n-1}}}^\perp$, $\forall 1\le i\le r$, we have from \cite{RH}: $\mathfrak{U}_{\dbl{v}_i}(\dbl{\eta}_i,\dbl{z}_i)=-\dbl{\eta}_i\dbl{v}_i^t\dbl{z}_i$. Thus, with respect to the parameterization that we consider we find that:
$$
\mathfrak{U}_y(\eta,z)=-\begin{pmatrix}0_r \\ \eta_1\dbl{v}_1^t\dbl{z}_1\\ \vdots \\ \eta_r\dbl{v}_r^t\dbl{z}_r\\ \eta\sp{\prime}_1\dbl{v}_1^t\dbl{z}_1\\ \vdots \\ \eta\sp{\prime}_r\dbl{v}_r^t\dbl{z}_r\end{pmatrix}.
$$
Let $G^R=(g_0; g_1; \dots; g_r; g\sp{\prime}_1; \dots; g\sp{\prime}_r)\in \R^{r+2nr}$ and $\dbl{g}_i=(g_i, g\sp{\prime}_i)$, $l_i=\dbl{v}_i\dbl{v}_i^t\dbl{g}_i$ for $i=1, \dots, r$. We obtain $\mathfrak{U}_y(\eta,\mathcal{P}_y^\perp G^R)$ by substituting $\dbl{z}_i$  by $l_i$ in $\mathfrak{U}_y(\eta,z)$. Since $\dbl{v}_i^t\dbl{v}_i=||\dbl{v}_i||^2=1$, we find that $\mathfrak{U}_y(\eta,\mathcal{P}_y^\perp G^R)=\begin{pmatrix}0_r \\ \eta_1 \dbl{v}_1^{t}\dbl{g}_1 \\ \vdots \\ \eta_r\dbl{v}_r^{t}\dbl{g}_r \\ \eta\sp{\prime}_1\dbl{v}_1^{t}\dbl{g}_1\\ \vdots \\ \eta\sp{\prime}_r \dbl{v}_r^{t}\dbl{g}_r\end{pmatrix}=S\eta$, where $S=\mathrm{diag}(0_{r\times r},\tilde{S},\tilde{S})$, with $\tilde{S}=\mathrm{diag}(s_1I_n, \dots, s_rI_n)$, $s_i=\dbl{v}_i^{t}\dbl{g}_i =\langle v_i, g_i\rangle+\langle v\sp{\prime}_i, g\sp{\prime}_i\rangle$. Since, $\mathfrak{U}_y(\eta,z)=\mathcal{P}_yD_{\eta}\mathcal{P}z$, and $\mathcal{P}_y\circ\mathcal{P}_y=\mathcal{P}_y$, we can write $\mathfrak{U}_y(\eta,z)=\mathcal{P}_y\mathfrak{U}_y(\eta,z)$. Hence, $\mathfrak{U}_y(\eta,\mathcal{P}_y^\perp G^R)=\mathcal{P}_yS\eta=\mathcal{P}_yS\mathcal{P}_y\eta$, since $\mathcal{P}_y\eta=\eta$ for $\eta\in T_y\N_r$. Thus we have: $H\eta=\mathcal{P}_y(H^R+S)\mathcal{P}_y\eta$, and then $H=\mathcal{P}_y(H^R+S)\mathcal{P}_y$. Herein, $H$ can be written with respect to the basis $Q$ as follows: $H=Q^t(H^R+S)Q$, which ends the proof.
%\end{proof}

\subsection{Real gradient and Hessian}\label{proofs}
In order to give the proofs of propositions \ref{gradient} and \ref{hessian}, we need the following discussion and auxiliary lemma.\\
We describe the real gradient and Hessian, by using complex variables and their conjugates. Recall from Brandwood \cite{1983IPCRS.130...11B} that transforming the pair $(\Re(z),\Im(z))$ of real and imaginary parts of a given complex variable $z$ into the pair $(z,\overline{z})$ is a simple linear transformation, which allows us to achieve explicit and simple computation of the gradient and Hessian of $f$.\\
Recall that $\cR_r=\left\{(W,\Re(V),\Im(V))\in \R^{r}\times \R^{n\times r} \times \R^{n\times r} \mid W\in\R^r, V\in\C^{n\times r}\right\}$, and that $f_R$ is the function $f$ seen as a function on  $\cR_{r}$.

Let $\cC_{r}=\left\{(W,V,\overline{V})\in \R^{r}\times \C^{n\times r}\times\C^{n\times r} \mid W\in\R^r,\right.$ $\left.V\in\C^{n\times r}\right\}$ and
\begin{equation}\label{eq1}
  K=\begin{bmatrix}I_{r}&0_{r\times 2nr}\\0_{2nr\times r}&J\end{bmatrix}
\end{equation}
where $J=\begin{bmatrix}I_{nr}&\ib I_{nr}\\I_{nr}&-\ib I_{nr}\end{bmatrix}$.
The linear map $K$ is an isomorphism between the $\R$-vector spaces $\cR_{r}$ and $\cC_{r}$. Its inverse is given by $K^{-1}=\begin{bmatrix}I_r&0_{r\times 2nr}\\0_{2nr\times r}&\frac{1}{2}J^*\end{bmatrix}$.

Let $f_C$ be the function $f$ seen as a function on  $\cC_{r}$. Considering $f_C$ for the computation of the gradient and the Hessian yields more elegant expressions than considering $f_R$. For this reason, we compute first the gradient and the Hessian of $f_C$, and then we use the isomorphism $K$ in \eqref{eq1} to get the real gradient and the Hessian of $f_R$.\\
\textbf{Lemma A.1.} \emph{The complex gradient $G^C$ can be transformed into the real gradient $G^R$ as follows:
    \begin{equation}\label{re1}
        G^R=K^tG^C.
    \end{equation}
    Similarly $H^R$ and $H^C$ are related by the following formula:
    \begin{equation}\label{er2}
        H^R=K^tH^CK.
    \end{equation}}

\begin{proof}
  See \cite{Sorber2012}  and the references therein.
\end{proof}
We can now present the proofs of propositions \ref{gradient} and \ref{hessian}.
\subsubsection{\textbf{Proof of proposition \ref{gradient}}}\label{proof2}
%\begin{proof}\label{proof}
We can write $f_C$ as
$f_C=\frac{1}{2}(f_1-f_2-f_3+f_4)$,
where
      \begin{equation*}
          \begin{split} f_1&=\Big|\Big|\sum_{i=1}^{r}{w_i(v_i^t\x)^d}\Big|\Big|_d^2
      =\sum_{|\alpha|=d}{\binom{d}{\alpha}\Big(\sum_{i=1}^{r}{w_i\bar{v}_i^\alpha}\Big)\Big(\sum_{i=1}^{r}{w_iv_i^\alpha}\Big)}~~~\text{(by \cref{norm}),}\\
      f_2&=\big\langle\sum_{i=1}^{r}{w_i(v_i^t\x)^d},\bm p\big\rangle_d=\sum_{i=1}^{r}{w_i\bm p(\bar{v}_i)}~~~\text{(by 1. in lemma \ref{ap}),}\\ f_3&=\bar{f_2}=\sum_{i=1}^{r}{w_i\bar{\bm p}(v_i)}, \text{and}~f_4=||\bm p||_d^2.
  \end{split}
\end{equation*}
Let us decompose $G^{C}$ as $G^C=\begin{pmatrix}G_1\\\tilde{G}_2\\\tilde{G}_3\end{pmatrix}$, with $G_1=(\frac{\partial f_C}{\partial w_j})_{1\le j\le r}$, $\tilde{G}_2=(\frac{\partial f_C}{\partial v_j})_{1\le j\le r}$ and $\tilde{G}_3=(\frac{\partial f_C}{\partial \overline{v}_j})_{1\le j\le r}$. As $f_C$ is a real valued function, we have that $\frac{\partial f_C}{\partial\bar{v}_j}=\overline{\frac{\partial f_C}{\partial v_j}}$  \cite{nehari1968introduction, remmert1991theory}, thus $\tilde{G}_3=\overline{\tilde{G}}_2$. Let us start by the computation of $G_1$:
\begin{equation*}
    \begin{split}
\frac{\partial f_1}{\partial w_j}&=\frac{\partial}{\partial w_j}\bigg(\sum_{|\alpha|=d}{\binom{d}{\alpha}\Big(\sum_{i=1}^{r}{w_i\bar{v}_i^\alpha}\Big)\Big(\sum_{i=1}^{r}{w_iv_i^\alpha}\Big)}\bigg)\\&=\sum_{|\alpha|=d}{\binom{d}{\alpha}\bigg(\bar{v}_j^\alpha\Big(\sum_{i=1}^{r}{w_iv_i^\alpha}\Big)+{v}_j^\alpha\Big(\sum_{i=1}^{r}{w_i\bar{v}_i^\alpha}\Big)\bigg)}\\&= \sum_{i=1}^{r}{w_i(v_j^*v_i)^d}+\sum_{i=1}^{r}{w_i(v_i^*v_j)^d}=2\sum_{i=1}^{r}{w_i\Re((v_j^*v_i)^d)};
\end{split}
\end{equation*} the third equality is deduced by using \cref{norm} and 1. of lemma \ref{ap}. In addition, we have
          $\frac{\partial f_2}{\partial w_j}=\frac{\partial}{\partial w_j}(\sum_{i=1}^{r}{w_i\bm p(\bar{v}_i)})=\bm p(\bar{v}_j)$,
          $\frac{\partial f_3}{\partial w_j}=\bar{\bm p}(v_j), \text{and~} \frac{\partial f_4}{\partial w_j}=0.$ Thus,
$\frac{\partial f_C}{\partial w_j}=\sum_{i=1}^{r}{w_i\Re((v_j^*v_i)^d)}-\Re(\bar{\bm p}(v_j)).$

Now, for the computation of $\tilde{G}_2$, let $\bm p=\sum_{|\alpha|=d}{\binom{d}{\alpha}\dbl{v}_\alpha \x^\alpha}$, and $1\le k\le n$,
\begin{equation*}
\begin{split}
\frac{\partial f_1}{\partial v_{j,k}}&=\sum_{|\alpha|=d}{\binom{d}{\alpha}\Big(\sum_{i=1}^{r}{w_i\bar{v}_i^\alpha}\Big)(w_j\alpha_k v_j^{\alpha-e_k})}
=w_j\sum_{i=1}^{r}{w_i\langle\partial_{x_k}(v_i^t\x)^d, (v_j^t\x)^{d-1}\rangle_{d-1}}\\
    &=dw_j\sum_{i=1}^{r}{w_i\langle (v_i^t\x)^d, x_k(v_j^t\x)^{d-1}\rangle_d}
    =dw_j\sum_{i=1}^{r}{w_i\bar{v}_{i,k}(v_i^*v_j)^{d-1}},
    \end{split}
\end{equation*}
    the second (resp. third and fourth) equality are deduced by using lemma \ref{ap}. Moreover, we have $\frac{\partial f_2}{\partial v_{j,k}}=0, \frac{\partial f_3}{\partial v_{j,k}}=w_j\sum_{|\alpha|=d}{\binom{d}{\alpha}\bar{\dbl{v}}_\alpha \alpha_k v_j^{\alpha-e_k}}=w_j\partial_{x_k}\bar{\bm p}(v_j), \text{and~}\frac{\partial f_4}{\partial v_{j,k}}=0.$ Thus, $\frac{\partial f_C}{\partial v_{j}}=\frac{1}{2}\Big(dw_j\sum_{i=1}^{r}{w_i({v}_i^*v_j)^{(d-1)}\bar{v}_i}-w_j\nabla_{\x}\bar{\bm p}(v_j)\Big)$.

We have $G^R=K^tG^C$ from \eqref{re1}. By multiplication of these two matrices, we obtain: $G^R=\begin{pmatrix}G_1\\\tilde{G}_2+\overline{\tilde{G}}_2\\\mathbf{i}(\tilde{G}_2-\overline{\tilde{G}}_2)\end{pmatrix}=\begin{pmatrix}G_1\\2\Re(\tilde{G}_2)\\-2\Im(\tilde{G}_2)\end{pmatrix}$. Finally dividing by 2, we get
$G^{R}= \begin{pmatrix}G_1\\\mathrm{\Re}(G_2)\\-\Im(G_2)\end{pmatrix}$, where $G_2=2\tilde{G}_2$, which ends the proof.
%\end{proof}

\subsubsection{\textbf{Proof of proposition \ref{hessian}}}\label{proof3}
%\begin{proof}
 $H^C$ is given by the following block matrix:
    \begin{equation*}
      H^C=\begin{bmatrix}\left[\frac{\partial^2 f_C}{\partial w_i\partial w_j}\right]_{1\le i,j\le r}&\left[\frac{\partial^2 f_C}{\partial w_i\partial v_j^t}\right]_{1\le i,j\le r}&\left[\frac{\partial^2 f_C}{\partial w_i\partial \bar{v}_j^t}\right]_{1\le i,j\le r}\\\left[\frac{\partial^2 f_C}{\partial v_i\partial w_j}\right]_{1\le i,j\le r}&\left[\frac{\partial^2 f_C}{\partial v_i\partial v_j^t}\right]_{1\le i,j\le r}&\left[\frac{\partial^2 f_C}{\partial v_i\partial \bar{v}_j^t}\right]_{1\le i,j\le r}\\\left[\frac{\partial ^2f_C}{\partial \bar{v}_i\partial w_j}\right]_{1\le i,j\le r}&\left[\frac{\partial^2 f_C}{\partial \bar{v}_i\partial v_j^t}\right]_{1\le i,j\le r}&\left[\frac{\partial^2 f_C}{\partial \bar{v}_i\partial \bar{v}_j^t}\right]_{1\le i,j\le r}\end{bmatrix}.
  \end{equation*}
  We have that $\frac{\partial^2 f}{\partial \bar{z}\partial \bar{z}^t}=\overline{\frac{\partial^2 f}{\partial z\partial z^t}}$, and $\frac{\partial^2 f}{\partial z\partial \bar{z}^t}=\frac{\partial^2 f}{\partial \bar{z}\partial z^t}$, for a complex variable $z$ and a real valued function with complex variables $f$. Using these two relations, we find that $\left[\frac{\partial^2 f_C}{\partial w_i\partial w_j}\right]_{1\le i,j\le r}$, $\left[\frac{\partial^2 f_C}{\partial v_i\partial w_j}\right]_{1\le i,j\le r}$, $\left[\frac{\partial^2 f_C}{\partial v_i\partial v_j^t}\right]_{1\le i,j\le r}$, and $\left[\frac{\partial f_C}{\partial \bar{v}_i\partial v_j^t}\right]_{1\le i,j\le r}$ determine $H^C$. We denote them respectively by $A$, $\tilde{B}$, $\tilde{C}$, and $\tilde{D}$. Herein, we can decompose $H^C$ as:
$$
H^C=\begin{bmatrix}A&{\tilde{B}}^t&{\tilde{B}}^*\\\tilde{B}&\tilde{C}&\tilde{D}^t\\\overline{\tilde{B}}&\tilde{D}&\overline{\tilde{C}}\end{bmatrix}.
   $$
   The computation of these four matrices can be done by taking the formula of $\frac{\partial f_C}{\partial w_j}$ and $\frac{\partial f_C}{\partial v_j}$ obtained in the proof of proposition \ref{gradient}, and using the apolar identities in lemma \ref{ap}. Using \eqref{er2} we obtain:
   $H^R=\begin{bmatrix}A&2{\Re(\tilde{B})}^t&-2{\Im(\tilde{B})}^t\\2\Re(\tilde{B})&2\Re(\tilde{C}+\tilde{D})&-2\Im(\tilde{C}+\tilde{D})\\-2\Im(\tilde{B})&2\Im(\tilde{D}-\tilde{C})&2\Re(\tilde{D}-\tilde{C})\end{bmatrix}$. Finally, for the simplification by 2, as in the previous proof, we redefine the formula of $H^R$ as it is given in proposition \ref{hessian}, where $B$, $C$, and $D$ are respectively equal to two times $\tilde{B}$,$\tilde{C}$, and $\tilde{D}$.

%\end{proof}

\end{document}

%% file: macro.tex
\usepackage{amsopn}

\DeclareMathOperator{\Hess}{Hess}
\DeclareMathOperator{\grad}{grad}
\def\ib{\mathbf{i}}

\def\C{\mathbb{C}}
\def\R{\mathbb{R}}
\def\V{\mathcal{V}}
\def\M{\mathcal{M}}

\def\B{{Q}}
\def\N{\mathcal{N}}
\def\O{\mathcal{O}}
\def\S{\mathcal{S}}
\def\T{\mathcal{T}}
\def\U{\mathcal{U}}
\def\cC{\mathcal{C}}
\def\cR{\mathcal{R}}
\def\x{\mathbf{x}}
\def\NN{\mathbb{N}}
\def\PP{\mathbb{P}}
\def\RR{\mathbb{R}}
\def\SS{\mathbb{S}}
\def\rank{\mathrm{rank}}

\newcommand{\apolar}[1]{\langle{#1}\rangle}
\newcommand{\bm}[1]{\mathbf{#1}}
\newcommand{\dbl}[1]{\check{#1}}

\def\rank{\mathrm{rank}}
\newcommand{\td}{\ensuremath{\mathcal{T}^d_{n}}}
\newcommand{\ts}{\ensuremath{\mathcal{S}^d_{n}}}
\newcommand{\hp}{\ensuremath{\mathbb{C}[\x]_{d}}}

\def\Vnd{\mathcal{V}^{n,d}}
\def\sec{\overline{\Sigma}}

%% file: header.tex
\usepackage{lipsum}
\usepackage{amsfonts}
\usepackage{graphicx}
\usepackage{epstopdf}
\usepackage{algorithmic}
\usepackage{multirow}
\usepackage{array}

\newcolumntype{?}{!{\vrule width 2pt}}

\makeatletter
\def\thickhline{%
  \noalign{\ifnum0=`}\fi\hrule \@height \thickarrayrulewidth \futurelet
   \reserved@a\@xthickhline}
\def\@xthickhline{\ifx\reserved@a\thickhline
               \vskip\doublerulesep
               \vskip-\thickarrayrulewidth
             \fi
      \ifnum0=`{\fi}}
\makeatother

\newlength{\thickarrayrulewidth}
\ifpdf
  \DeclareGraphicsExtensions{.eps,.pdf,.png,.jpg}
\else
  \DeclareGraphicsExtensions{.eps}
\fi

\makeatletter
\newcommand*{\addFileDependency}[1]{% argument=file name and extension
  \typeout{(#1)}% latexmk will find this if $recorder=0 (however, in that case, it will ignore #1 if it is a .aux or .pdf file etc and it exists! if it doesn't exist, it will appear in the list of dependents regardless)
  \@addtofilelist{#1}% if you want it to appear in \listfiles, not really necessary and latexmk doesn't use this
  \IfFileExists{#1}{}{\typeout{No file #1.}}% latexmk will find this message if #1 doesn't exist (yet)
}
\makeatother

\theoremstyle{definition}
\newtheorem{definition}{Definition}[section]

\newtheorem{remark}[definition]{Remark}
\newtheorem{example}[definition]{Example}

\theoremstyle{plain}
\newtheorem{lemma}[definition]{Lemma}
\newtheorem{proposition}[definition]{Proposition}